\documentclass{article}
\usepackage[margin=0.5in]{geometry}
\usepackage[utf8]{inputenc}
\usepackage{fullpage}
\usepackage{enumerate}
\usepackage{authblk}
\usepackage[colorinlistoftodos]{todonotes}
\usepackage{verbatim}%comments out anything between \begin{comment} and \end{comment}
\usepackage{section, amsthm, textcase, setspace, amssymb, lineno, amsmath, amssymb, amsfonts, latexsym, fancyhdr, longtable, ulem}
\usepackage{epsfig, graphicx, pstricks,pst-grad,pst-text,tikz, colortbl}
\usepackage{epsf}
\usepackage{graphicx, color}
\usepackage{float}
\usepackage[rflt]{floatflt}
\usepackage{amsfonts, mathrsfs}
\usepackage{latexsym}
\usetikzlibrary{fit,matrix,positioning}
\usepackage{pdflscape}
\usepackage{kbordermatrix, enumitem}
\usepackage[most]{tcolorbox}

\newtheorem{lemma}{Lemma}
\newtheorem{theorem}{Theorem}
\newtheorem{definition}{Definition}
\newtheorem{Ex}{Example}
\newtheorem{remark}{Remark}
\newtheorem{corollary}{Corollary}
\newtheorem*{theorem*}{Theorem}

\newcommand{\ind}{{\rm ind \hspace{.1cm}}}

\newcommand\addvmargin[1]{
  \node[fit=(current bounding box),inner ysep=#1,inner xsep=0]{};}

\begin{document}

\title{Contact Lie poset algebras}

\author[*]{Vincent E. Coll, Jr.}
\author[**]{Nicholas Mayers}
\author[*]{Nicholas Russoniello}

\affil[*]{Department of Mathematics, Lehigh University, Bethlehem, PA, 18015}
\affil[**]{Department of Mathematics, Milwaukee School of Engineering, Milwaukee, WI, 53202}

\maketitle

\begin{center}
    \textbf{Dedication}\\
    
    In honor of Murray Gerstenhaber -- recipient of the 2021 Steele Prize for seminal contribution to research
    
\end{center}

\bigskip
\begin{abstract} 
\noindent
We provide a combinatorial recipe for constructing all posets of height at most two for which the corresponding type-A Lie poset algebra is contact. In the case that such posets are connected, a discrete Morse theory argument establishes that the posets' simplicial realizations are contractible.  It follows from a cohomological result of Coll and Gerstenhaber on Lie semi-direct products that the corresponding contact Lie algebras are absolutely rigid. 
\end{abstract}

%\linenumbers

%%%%%%%%%%%%%%%%%%%%%%%%%%%%%%%%%%%
%%%%%%%%%%%%%%%%%%%%%%%%%%%%%%%%%%%
\section{Introduction}
%%%%%%%%%%%%%%%%%%%%%%%%%%%%%%%%%%%
%%%%%%%%%%%%%%%%%%%%%%%%%%%%%%%%%%%

This article is a direct follow-up to the article \textit{The index of Lie poset algebras} (see \textbf{\cite{seriesA}},
J. Comb. Theory Ser. A, 2021).  In that article,  the authors developed ``user-friendly'' combinatorial index formulas for type-A Lie poset algebras and focused on those which are Frobenius -- having index zero.\footnote{The index of a Lie algebra is an algebraic invariant and is defined as follows.  If $\mathfrak{g}^*$ is the set of linear one-forms on $\mathfrak{g}$ then

\[
\ind \mathfrak{g}=\min_{\varphi\in \mathfrak{g}^*} \dim  (\ker (B_\varphi)),
\]
where $B_\varphi$ is the associated skew-symmetric \textit{Kirillov form} defined by $B_\varphi(x,y)=\varphi([x,y])$, for all  $x,y\in\mathfrak{g}$. Algebras with index zero are called \textit{Frobenius} and are of interest to those working in invariant theory \textbf{\cite{Ooms}} and deformation theory owing to their connection with constant solutions to the classical Yang-Baxter equation (see \textbf{\cite{G1}} and \textbf{\cite{G2}}). } In particular, they established that the associated \textit{Frobenius poset} is the iterative limit of a sequence of posets built up from building-block posets using gluing rules.  This leads to a constructive characterization of Frobenius, type-A Lie poset algebras where the chains of the associated poset have cardinality at most three. 

%Such algebras are absolutely rigid as a consequence of a cohomological result of Coll and Gerstenhaber (Theorem 5.6, \textbf{\cite{CG}}).

Here, we are concerned with contact, type-A Lie poset algebras. Formally, an odd-dimensional Lie algebra $\mathfrak{g}$ is \textit{contact} if 
there exists $\varphi\in\mathfrak{g}^*$ such that $\varphi  \wedge (d\varphi)^n\ne 0$, where dim $\mathfrak{g}=2n+1$. The one-form
$\varphi$  is called a \textit{contact form} (or \textit{contact structure}).  The 
$(2n+1)$-form $\varphi  \wedge (d\varphi)^n\ne 0$ is a \textit{volume form} on the underlying Lie group.  The construction and classification of contact manifolds is a central problem in differential topology (see \textbf{\cite{wein}}).

The twofold goal of this article is to characterize contact, type-A Lie poset algebras whose associated \textit{contact posets} have chains of cardinality at most three and to show that such algebras are absolutely rigid.

%Assuming the underlying poset is connected, we find, as in the Frobenius case, that these algebras are absolutely rigid -- again as a consequence of Theorem 5.6 in \textbf{\cite{CG}}.

Contact Lie algebras have index one, but this is generally not a sufficient condition -- even in the type-A Lie poset setting.\footnote{The Lie algebra $\mathfrak{g}=\langle e_1,e_2,e_3,e_4,e_5,e_6,e_7\rangle$ with relations $[e_1,e_4]=2e_4$, $[e_2,e_4]=e_4$, $[e_1,e_5]=e_5$, $[e_2,e_5]=2e_5$, $[e_3,e_5]=e_5$, $[e_1,e_6]=e_6$, $[e_3,e_6]=e_6$, $[e_2,e_7]=e_7$ and $[e_3,e_7]=2e_7$ is a type-A Lie poset algebra which has index one but is not contact.} 
To find contact, type-A Lie poset algebras, we leverage the index formulas of \textbf{\cite{seriesA}} to identify index-one algebras which are then contact building-block ``candidates.'' From there, 
using a modification of the iterative process for building Frobenius, type-A Lie poset algebras outlined in \textbf{\cite{seriesA}}, we find that
a poset is contact if and only if it is the recursive limit of a sequence of posets constructed using an updated collection of building-block posets and gluing rules. More care must be taken than in the Frobenius case, where it is enough to insure that, during the construction process, the index remains zero.  Here we must also keep track of the evolving contact form.\footnote{A similar limiting process is used in \textbf{\cite{CMToral}} to construct Frobenius (index-realizing) forms on the class of ``toral," Frobenius, type-A Lie poset algebras. The posets underlying toral, type-A Lie poset algebras arise from a generalization of the constructive procedure outlined for building Frobenius posets in \textbf{\cite{seriesA}}.}  This is the first of two principal results and is the main combinatorial result of this paper (see Theorem \ref{thm:contact}).

A discrete Morse theory argument establishes that the simplicial complex associated with any connected, contact poset with chains of cardinality at most three is contractible, so has no simplicial homology (see Theorem~\ref{Nohomology}). A recent result of Coll and Gerstenhaber (Theorem~\ref{CG}) can then be applied to find that the second Lie cohomology group of the corresponding type-A Lie poset algebra with coefficients in itself is zero. 
This yields the second main result of this paper (Theorem \ref{thm:main2}) and mirrors an analogous rigidity result for  Frobenius, type-A Lie poset algebras (see \textbf{\cite{seriesA}}, Theorem 16). 

\bigskip
\noindent
\textbf{Theorem.}
\textit{A contact, type-A Lie poset algebra corresponding to a connected poset of height zero, one, or two is absolutely rigid.}

%\begin{theorem*} 
%A contact \textup(or Frobenius\textup), type-A Lie poset algebra corresponding to a poset of height zero, one, or two is absolutely rigid.
%\end{theorem*}

%The index formula of Theorem~\ref{thm:gind} subsequently yields a characterization of posets of height at most two, which are associated to Frobenius Lie poset algebras.  In heights one and two, this characterization takes the form of a combinatorial recipe -- building blocks and gluing rules -- for the construction of all such posets (see Theorem~\ref{thm:h2frobchar}). A discrete Morse theory argument then establishes that the simplicial complex associated with any such poset $\mathcal{P}$ is contractible, so has no simplicial homology (see Theorem~\ref{Nohomology}). A recent result of Coll and Gerstenhaber (Theorem~\ref{CG}) can then be applied to find that the second Lie cohomology group of the corresponding type-A Lie poset algebra with coefficients in  
%itself is zero.\footnote{Theorem~\ref{CG} is the Lie algebraic analogue of the now classical result of Gerstenhaber and Schack which asserts that simplicial cohomology is a special case of  Hochschild cohomology \textbf{\cite{G3}}.  }

\section{Preliminaries}

A \textit{finite poset} $(\mathcal{P}, \preceq_{\mathcal{P}})$ consists of a finite set $\mathcal{P}=\{1,\hdots,n\}$ together with a binary relation $\preceq_{\mathcal{P}}$ which is reflexive, anti-symmetric, and transitive. It is further assumed that if $x\preceq_{\mathcal{P}}y$ for $x,y\in\mathcal{P}$, then $x\le y$, where $\le$ denotes the natural ordering on $\mathbb{Z}$. When no confusion will arise, we simply denote a poset $(\mathcal{P}, \preceq_{\mathcal{P}})$ by $\mathcal{P}$, and $\preceq_{\mathcal{P}}$ by $\preceq$. 

Let $x,y\in\mathcal{P}$. If $x\preceq y$ and $x\neq y$, then we call $x\preceq y$ a \textit{strict relation} and write $x\prec y$. Let $Rel(\mathcal{P})$ denote the set of strict relations between elements of $\mathcal{P}$, $Ext(\mathcal{P})$ denote the set of minimal and maximal elements of $\mathcal{P}$, and $Rel_E(\mathcal{P})$ denote the set of strict relations between the elements of $Ext(\mathcal{P})$.

\begin{Ex}\label{ex:stargate}
Consider the poset $\mathcal{P}=\{1,2,3,4\}$ with $1\prec2\prec3,4$.  We have that $$Rel(\mathcal{P})=\{1\prec 2,1\prec 3,1\prec 4,2\prec 3,2\prec 4\},$$  $$Ext(\mathcal{P})=\{1,3,4\},\quad \text{and}\quad Rel_E(\mathcal{P})=\{1\prec 3, 1\prec 4\}.$$
\end{Ex}

\noindent
Recall that, if $x\prec y$ and there exists no $z\in \mathcal{P}$ satisfying $x\prec z\prec y$, then $y$ \textit{covers} $x$ and 
$x\prec y$ is a \textit{covering relation}.  Using this language, the \textit{Hasse diagram} of a poset $\mathcal{P}$ can be reckoned as the graph whose vertices correspond to elements of $\mathcal{P}$ and whose edges correspond to covering relations. A poset $\mathcal{P}$ is \textit{connected} if the Hasse diagram of $\mathcal{P}$ is connected as a graph, and \textit{disconnected} otherwise. Throughout this paper, $C_{\mathcal{P}}$ will denote the number of connected components of the Hasse diagram of $\mathcal{P}$.

\begin{Ex}
Let $\mathcal{P}$ be the poset of Example~\ref{ex:stargate}. The Hasse diagram of $\mathcal{P}$ is given below in Figure~\ref{fig:Hasse}.
\begin{figure}[H]
$$\begin{tikzpicture}
	\node (1) at (0, 0) [circle, draw = black, fill = black, inner sep = 0.5mm, label=left:{1}]{};
	\node (2) at (0, 1)[circle, draw = black, fill = black, inner sep = 0.5mm, label=left:{2}] {};
	\node (3) at (-0.5, 2) [circle, draw = black, fill = black, inner sep = 0.5mm, label=left:{3}] {};
	\node (4) at (0.5, 2) [circle, draw = black, fill = black, inner sep = 0.5mm, label=right:{4}] {};
    \draw (1)--(2);
    \draw (2)--(3);
    \draw (2)--(4);
    \addvmargin{1mm}
\end{tikzpicture}$$
\caption{Hasse diagram of $\mathcal{P}$}\label{fig:Hasse}
\end{figure}
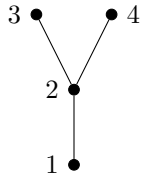
\end{Ex}

Given a subset $S\subset\mathcal{P}$, the \textit{induced subposet generated by $S$} is the poset $\mathcal{P}_S$ on $S$, where $i\preceq_{\mathcal{P}_S}j$ if and only if $i\preceq_{\mathcal{P}}j$. A totally ordered subset $S\subset\mathcal{P}$ is called a \textit{chain}. Using the chains of $\mathcal{P}$, one can define a simplicial complex $\Sigma(\mathcal{P})$ by having chains of cardinality $n$ in $\mathcal{P}$ define the $(n-1)$-dimensional faces of $\Sigma(\mathcal{P})$. A chain $S\subset\mathcal{P}$ is called \textit{maximal} if it is not a proper subset of any other chain $S'\subset \mathcal{P}$. If every maximal chain of a poset $\mathcal{P}$ is of the same cardinality, then we call $\mathcal{P}$ \textit{pure}. When a poset is pure, there is a natural grading on the elements of $\mathcal{P}$. This grading is made precise by a rank function $r:\mathcal{P}\to \mathbb{Z}_{\ge 0}$, where minimal elements have rank zero, and if $x$ is covered by $y$ in $\mathcal{P}$, then $r(y)=r(x)+1$. Note that the poset of Example 1 is pure since its maximal chains $1\prec 2\prec 3$ and $1\prec 2\prec 4$ both have cardinality three.  This poset has a single minimal element of rank zero, namely $\{1\}$, a single element of rank one, namely $\{2\}$, and two maximal elements of rank two, namely $\{3,4\}$. We define the \textit{height} of a poset $\mathcal{P}$ to be one less than the largest cardinality of a chain in $\mathcal{P}$.

The following family of posets and poset operation will be important in the sections that follow.

\begin{definition}\label{def:comppos}
Let $\mathcal{P}$ be the pure poset with $r_i$ elements of rank $i$, for $0\le i\le n$, and every possible relation between elements of differing rank. We denote such ``complete" posets by $\mathcal{P}(r_0, r_1,\hdots, r_n)$. 
\end{definition}

\begin{Ex}
Using the notation of Definition~\ref{def:comppos}, the poset of Example~\ref{ex:stargate} is $\mathcal{P}(1,1,2)$.
\end{Ex}

\begin{definition}
Given two posets $\mathcal{P}$ and $\mathcal{Q}$, the \textit{disjoint sum} of $\mathcal{P}$ and $\mathcal{Q}$ is the poset $\mathcal{P}+\mathcal{Q}$ on the disjoint sum of $\mathcal{P}$ and $\mathcal{Q}$, where $s\preceq_{\mathcal{P}+\mathcal{Q}} t$ if either 
\begin{itemize}
    \item $s,t\in \mathcal{P}$ and $s\preceq_{\mathcal{P}} t$, or
    \item $s,t\in \mathcal{Q}$ and $s\preceq_{\mathcal{Q}} t$.
\end{itemize}
\end{definition}

Let $\mathcal{P}$ be a finite poset and \textbf{k} be an algebraically closed field of characteristic zero, which we may take to be the complex numbers. The (associative) \textit{incidence algebra} $A(\mathcal{P})=A(\mathcal{P}, \textbf{k})$ is the span over $\textbf{k}$ of elements $e_{i,j}$, for $i,j\in\mathcal{P}$ satisfying $i\preceq j$, with product given by setting $e_{i,j}e_{kl}=e_{i,l}$ if $j=k$ and $0$ otherwise. The \textit{trace} of an element $\sum c_{i,j}e_{i,j}$ is $\sum c_{i,i}.$

We can equip $A(\mathcal{P})$ with the commutator product $[a,b]=ab-ba$, where juxtaposition denotes the product in $A(\mathcal{P})$, to produce the \textit{Lie poset algebra} $\mathfrak{g}(\mathcal{P})=\mathfrak{g}(\mathcal{P}, \textbf{k})$. If $|\mathcal{P}|=n$, then both $A(\mathcal{P})$ and $\mathfrak{g}(\mathcal{P})$ may be regarded as subalgebras of the algebra of $n \times n$ upper-triangular matrices over $\textbf{k}$. Such a matrix representation is realized by replacing each basis element $e_{i,j}$ by the $n\times n$ matrix $E_{i,j}$ containing a 1 in the $i,j$-entry and 0's elsewhere. The product between elements $e_{i,j}$ is then replaced by matrix multiplication between the $E_{i,j}$.

\begin{Ex}\label{ex:posetmat}
Let $\mathcal{P}$ be the poset of Example~\ref{ex:stargate}. The matrix form of elements in $\mathfrak{g}(\mathcal{P})$ is illustrated in Figure~\ref{fig:tA}, where the $*$'s denote potential non-zero entries. 
\begin{figure}[H]
$$\kbordermatrix{
    & 1 & 2 & 3 & 4  \\
   1 & * & * & * & *   \\
   2 & 0 & * & * & *  \\
   3 & 0 & 0 & * & 0  \\
   4 & 0 & 0 & 0 & *  \\
  }$$
\caption{Matrix form of $\mathfrak{g}(\mathcal{P})$, for $\mathcal{P}=\{1,2,3,4\}$ with $1\prec 2\prec 3,4$}\label{fig:tA}
\end{figure}
\end{Ex}

\noindent
Restricting $\mathfrak{g}(\mathcal{P})$ to trace-zero matrices yields a subalgebra of the first classical family $A_{n-1}=\mathfrak{sl}(n)$.  We denote the resulting \textit{type-A Lie poset algebra} by $\mathfrak{g}_A(\mathcal{P})$. 

In \textbf{\cite{seriesA}}, the authors establish combinatorial index formulas for type-A Lie poset algebras, which are recorded in Theorem~\ref{thm:gind} below. We require a preliminary definition.

\begin{definition}
Let $\mathcal{P}$ be a poset and $j\in \mathcal{P}$.  Define 

$$D(\mathcal{P},j)=|\{i\in \mathcal{P}~|~i\prec j\}|,$$ 

$$U(\mathcal{P},j)=|\{i\in \mathcal{P}~|~j\prec i\}|,$$ and 

$$UD(\mathcal{P},j) =  \begin{cases} 
      |U(\mathcal{P},j)-D(\mathcal{P},j)|, & U(\mathcal{P},j)\neq D(\mathcal{P},j); \\
      2, & \text{otherwise.}
   \end{cases}
$$
\end{definition}

\begin{theorem}[Coll and Mayers \textbf{\cite{seriesA}}, Theorem 9]\label{thm:gind}
If $\mathcal{P}$ is a poset of height at most two, then 
$$\ind\mathfrak{g}_A(\mathcal{P})=|Rel_E(\mathcal{P})|-|\mathcal{P}|+2\cdot C_{\mathcal{P}}-1+\sum_{j\in \mathcal{P}\backslash Ext(\mathcal{P})}UD(\mathcal{P},j).$$
\end{theorem}

\begin{theorem}[Coll and Mayers \textbf{\cite{seriesA}}, Theorem 12]\label{thm:h2frob}
A poset $\mathcal{P}$ of height at most two is Frobenius if and only if
\begin{itemize}
    \item $|Ext(\mathcal{P}_{\{j\in\mathcal{P}~|~i\preceq j\text{ or }j\preceq i\}})|=3$ for all $i\in\mathcal{P}\backslash Ext(\mathcal{P})$, and
    \item the Hasse diagram of $\mathcal{P}_{Ext(\mathcal{P})}$ is a tree.
\end{itemize}
\end{theorem}

\section{Contact Lie algebras}

In this section, we establish some standard notation and provide several structural results regarding contact Lie algebras. 

Let $\mathfrak{g}$ be an $n$-dimensional Lie algebra with ordered basis $\mathscr{B}(\mathfrak{g})=\{E_1,\dots,E_n\}$, and define $$C(\mathfrak{g},\mathscr{B}(\mathfrak{g}))=([E_i,E_j])_{1\leq i,j\leq n}$$ to be the \textit{commutator matrix} associated with $\mathfrak{g}$. Now, for any $\varphi\in\mathfrak{g}^*$, define the matrix
  $$[B_{\varphi}]=\varphi\left(C(\mathfrak{g},\mathscr{B}(\mathfrak{g}))\right)=(\varphi([E_i,E_j]))_{1\leq i,j\leq n}.$$  
  
\begin{remark}
Given a Lie algebra $\mathfrak{g}$ such that $\ind\mathfrak{g}=k$, a one-form $\varphi\in\mathfrak{g}^*$ is called regular if it is index realizing; that is, if $\dim\ker B_{\varphi}=\ind\mathfrak{g}$. Note that $\varphi\in\mathfrak{g}^*$ is regular if and only if $\text{rank}\left([B_{\varphi}]\right)=\dim\mathfrak{g}-\ind\mathfrak{g}$.
\end{remark}
  
\noindent
Recall that $\mathfrak{g}$ is contact only if it is odd-dimensional, so let $n=2k+1.$ Let $[I]^t=(E_1\hdots E_{2k+1})$ and define $$\widehat{C}(\mathfrak{g},\mathscr{B}(\mathfrak{g}))=\begin{bmatrix}
0 & [I]^t\\
-[I] & C(\mathfrak{g},\mathscr{B}(\mathfrak{g}))
\end{bmatrix}.$$  
If $\{E_1^*,\dots,E_{2k+1}^*\}$ is the ``dual basis" associated to $\mathscr{B}(\mathfrak{g}),$ then $\varphi$ can be written as a linear combination\linebreak $\varphi=\sum_{i=1}^{2k+1}x_iE_i^*.$ In vector notation, $[\varphi]=(x_1,\dots,x_{2k+1})^t.$ Applying $\varphi$ to each entry of $\widehat{C}(\mathfrak{g},\mathscr{B}(\mathfrak{g}))$ yields the $(2k+2)$-dimensional skew-symmetric matrix $$\left[\widehat{B}_{\varphi}\right]=\varphi\left(\widehat{C}(\mathfrak{g},\mathscr{B}(\mathfrak{g}))\right)=\begin{bmatrix}
0 & [\varphi]^t\\
-[\varphi] & [B_{\varphi}]
\end{bmatrix}.$$    

%\begin{remark}
%Throughout the remainder of this paper, given a Lie algebra $\mathfrak{g}$, unless otherwise stated $\mathscr{B}(\mathfrak{g})$ will represent an arbitrary fixed ordered bases of $\mathfrak{g}$.
%\end{remark}

\noindent
Straightforward computations give the following convenient characterization of contact Lie algebras.

\begin{theorem}[Salgado \textbf{\cite{Sally}}]\label{thm:det} Let $\mathfrak{g}$ be an $n$-dimensional Lie algebra with $\varphi\in\mathfrak{g}^*$. If $n$ is odd, then $\mathfrak{g}$ is contact with contact form $\varphi$ if and only if $\det\left(\left[\widehat{B}_{\varphi}\right]\right)\neq 0$.
\end{theorem}

Using Theorem~\ref{thm:det},  we are able to establish the following structural results related to contact Lie algebras which will be crucial in what follows. Let $Z(\mathfrak{g})$ denote the center of a Lie algebra $\mathfrak{g}$.

\begin{theorem}\label{lem:commute}
Let $\mathfrak{g}$ be a Lie algebra. If $\ind\mathfrak{g}=1$ and $\dim Z(\mathfrak{g})>0$, then $\mathfrak{g}$ is contact.
\end{theorem}
\begin{proof}
We first establish notation that we will use in this proof -- and ongoing.  

\bigskip
\noindent
\textit{Notation:} Given a functional $\varphi\in \mathfrak{g}^*$. We refer to the first column \textup(resp., row\textup) of $\left[\widehat{B}_{\varphi}\right]$ as column \textup(resp., row\textup) $\mathbf{I}$ and the column \textup(resp., row\textup) of $\left[\widehat{B}_{\varphi}\right]$ corresponding to the column \textup(resp., row\textup) of $\widehat{C}(\mathfrak{g},\mathscr{B}(\mathfrak{g}))$ with first entry $b_i$ (resp., $-b_i$) as column \textup(resp., row\textup) $\mathbf{b_i}$.

\bigskip
Let $z\in Z(\mathfrak{g})$. We claim that there exists a regular $\varphi\in\mathfrak{g}^*$ for which $\varphi(z)\neq 0$. To see this, extend $z$ to a basis $\mathscr{B}(\mathfrak{g})$ of $\mathfrak{g}$, and note that for any $\varphi'\in\mathfrak{g}^*$ the row corresponding to $z$ in $[B_{\varphi'}]=\varphi'(C(\mathfrak{g},\mathscr{B}(\mathfrak{g})))$ is the zero row. Thus, since $\ind\mathfrak{g}=1$, if $[B'_{\varphi'}]$ denotes the matrix formed by removing the row and column of $[B_{\varphi'}]$ corresponding to $z$, then $\det \left(\left[B'_{\varphi'}\right]\right)\neq 0$ if and only if $\varphi'$ is a regular one-form on $\mathfrak{g}$. Fix a regular $\varphi'\in\mathfrak{g}^*,$ and replace $\varphi'(z)$ in $\left[B'_{\varphi'}\right]$ with a variable $x.$ Then $\det\left(\left[B'_{\varphi'}\right]\right)$ is a nonzero polynomial $p(x).$ Since $\mathbf{k}$ is an infinite field and $p(x)$ only has a finite number of zeros, there exists a nonzero choice $v\in \mathbf{k}$ for which $p(v)\neq 0$. Let $\varphi\in\mathfrak{g}^*$ be defined by $\varphi(b)=\varphi'(b)$, for all $b\in \mathscr{B}(\mathfrak{g})\backslash\{z\}$, and $\varphi(z)=v\neq 0$. By construction, $\det \left(\left[B'_{\varphi}\right]\right)\neq 0$ and so $\varphi$ is regular, establishing the claim. Now, consider $\det\left(\left[\widehat{B}_{\varphi}\right]\right)$. Expanding the determinant along row $\mathbf{z}$ followed by column $\mathbf{z},$ we find that $$\det\left(\left[\widehat{B}_{\varphi}\right]\right) = \varphi(z)^2\det \left(\left[B'_{\varphi}\right]\right)\neq 0.$$ An application of Theorem~\ref{thm:det} establishes that $\mathfrak{g}$ is contact with contact form $\varphi$.
\end{proof}

The final result of this section is specific to contact, type-A Lie poset algebras. A Hasse diagram corresponding to a poset has upward orientation and contains no directed cycles. However, in an abuse of terminology, we will frequently use ``cycles" in a Hasse diagram to describe paths in the undirected Hasse diagram which begin and end at the same vertex. See Example~\ref{ex:cycle}.

\begin{Ex}\label{ex:cycle}
Consider the Hasse diagrams of the following posets: $\mathcal{P}_1=\{1,2,3,4,5,6\}$ with relations $1\prec 3,4\prec 5$ and $2\prec 4\prec 6$ \textup(Figure \ref{fig:cycles} left\textup), and $\mathcal{P}_2=\{1,2,3,4\}$ with relations $1,2\prec 3,4$ \textup(Figure \ref{fig:cycles} right\textup).
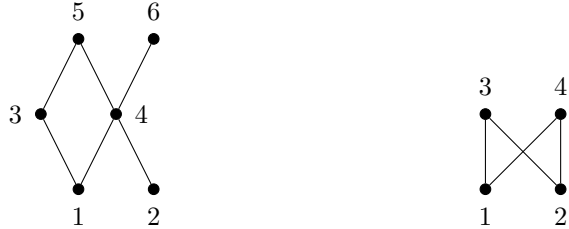
\begin{figure}[H]\label{Hasse}
$$\begin{tikzpicture}
\node (1) at (0,0) [circle, draw = black, fill = black, inner sep = 0.5mm] {};
\node (2) at (1,0) [circle,draw=black, fill=black, inner sep=0.5mm]{};
\node (3) at (-0.5,1) [circle,draw=black, fill=black, inner sep=0.5mm]{};
\node (4) at (0.5,1) [circle,draw=black, fill=black, inner sep=0.5mm]{};
\node (5) at (0,2) [circle,draw=black, fill=black, inner sep=0.5mm]{};
\node (6) at (1,2) [circle,draw=black, fill=black, inner sep=0.5mm]{};
\node[label=below:{$1$}] at (0,0){};
\node[label=below:{$2$}] at (1,0){};
\node[label=left:{$3$}] at (-0.5,1){};
\node[label=right:{$4$}] at (0.5,1){};
\node[label=above:{$5$}] at (0,2){};
\node[label=above:{$6$}] at (1,2){};
\draw (1)--(3);
\draw (1)--(4);
\draw (3)--(5);
\draw (4)--(5);
\draw (4)--(6);
\draw (2)--(4);
\end{tikzpicture}
\hspace{4cm}
\begin{tikzpicture}
\node (1) at (0,0) [circle, draw = black, fill = black, inner sep = 0.5mm] {};
\node (2) at (1,0) [circle,draw=black, fill=black, inner sep=0.5mm]{};
\node (3) at (0,1) [circle,draw=black, fill=black, inner sep=0.5mm]{};
\node (4) at (1,1) [circle,draw=black, fill=black, inner sep=0.5mm]{};
\node[label=above:{$4$}] at (1,1){};
\node[label=above:{$3$}] at (0,1){};
\node[label=below:{$2$}] at (1,0){};
\node[label=below:{$1$}] at (0,0){};
\draw (1)--(3);
\draw (1)--(4);
\draw (2)--(3);
\draw (2)--(4);
\end{tikzpicture}$$
    \caption{$\mathcal{P}_1$ contains cycle $(1,3,5,4)$ and $\mathcal{P}_2$ contains cycle $(1,4,2,3)$}
    \label{fig:cycles}
\end{figure}
\end{Ex}

Using the above notion of cycles, we have the following obstruction theorem.

\begin{theorem}\label{lem:cycle}
If the Hasse diagram of $\mathcal{P}_{Ext(\mathcal{P})}$ contains a cycle, then $\mathfrak{g}_A(\mathcal{P})$ is not contact.
\end{theorem}
\begin{proof}
Assume, for a contradiction, that $\mathfrak{g}=\mathfrak{g}_A(\mathcal{P})$ is contact and that $\mathcal{P}_{Ext(\mathcal{P})}$ contains a cycle $\Gamma$ with edge set $E=\{e_1,\dots,e_n\}$. Let $\varphi\in \mathfrak{g}^*$ be a contact form on $\mathfrak{g}$, $$\mathscr{B}(\mathfrak{g})=\{E_{1,1}-E_{p,p}~|~1\neq p\in\mathcal{P}\}\cup\{E_{p,q}~|~p\prec q\},$$ and $\left[\widehat{B}_{\varphi}\right]=\varphi\left(\widehat{C}(\mathfrak{g},\mathscr{B}(\mathfrak{g}))\right)$. By assumption, $\det\left(\left[\widehat{B}_{\varphi}\right]\right)\neq 0$. Partition $E$ into two sets $$E_1 = \{e_i~|~i\text{ odd}\}\quad{ and }\quad E_2 = \{e_i~|~i\text{ even}\}.$$ 

\noindent
\textit{Observation}. Basis elements of the form $E_{i,j}$, for $(i,j)\in E$, commute with all basis elements of the form $E_{p,q}$, for $p,q\in\mathcal{P}$ satisfying $p\prec q$. Thus, such rows have nonzero entries only in columns of the form $\mathbf{I}$ and $\mathbf{E_{1,1} - E_{k,k}}$; in particular, all nonzero entries must be $\pm\varphi(E_{i,j})$ or $\pm2\varphi(E_{i,j})$. Hence, as $\det\left(\left[\widehat{B}_{\varphi}\right]\right)$ is assumed to be nonzero, it follows that $\varphi(E_{i,j})\neq 0$, for all $(i,j)\in E$.

\bigskip
Now, consider the following linear combination of rows of $\left[\widehat{B}_{\varphi}\right]$: $$\mathscr{L}=\sum_{(i,j)\in E_1}\frac{1}{\varphi(E_{i,j})}\mathbf{E_{i,j}} - \sum_{(i,j)\in E_2}\frac{1}{\varphi(E_{i,j})}\mathbf{E_{i,j}}.$$ We claim that $\mathscr{L}$ is equal to the zero-row.

To establish the claim, in light of the observation above, it suffices to consider the entries of $\mathscr{L}$ in columns of the form $\mathbf{I}$ and $\mathbf{E_{1,1} - E_{k,k}}$. There are four cases.
\\*

\noindent
\textbf{Case 1}: Column $\mathbf{I}$. Recall that column $\mathbf{I}$ has an entry of $-\varphi(E_{p,q})$ in row $\mathbf{E_{p,q}}$, for $p,q\in\mathcal{P}$ satisfying $p\prec q$. Thus, $\mathscr{L}$ has an entry of 
\begin{equation}\label{caseI}
\sum_{(i,j)\in E_2}1-\sum_{(i,j)\in E_1}1=|E_2|-|E_1|
\end{equation}
in column $\mathbf{I}$. Since $\mathcal{P}_{Ext(\mathcal{P})}$ is a bipartite graph, $\Gamma$ is an even cycle. Therefore, $|E_1|=|E_2|,$ so the difference in (\ref{caseI}) is equal to 0.
\\*

\noindent
\textbf{Case 2}: Columns of the form $\mathbf{E_{1,1} - E_{k,k}}$ where $(1,k)$ is in $E$. In this case, row $\mathbf{E_{1,k}}$ contributes $\pm 2$ to column $\mathbf{E_{1,1} - E_{k,k}}$ of $\mathscr{L}$. Further, since $(1,k)$ is an edge of $\Gamma$, it must be adjacent to two other edges of $\Gamma$, say $(i,k)$ and $(1,j)$. Note that rows $\mathbf{E_{i,k}}$ and $\mathbf{E_{1,j}}$ both contribute $\mp 1$ to column $\mathbf{E_{1,1} - E_{k,k}}$ of $\mathscr{L}$. As no other rows involved in $\mathscr{L}$ have nonzero entries in column $\mathbf{E_{1,1}-E_{k,k}}$, we conclude that $\mathscr{L}$ has an entry of 0 in this column.
\\*

\noindent
\textbf{Case 3}: Columns of the form $\mathbf{E_{1,1} - E_{k,k}}$ where $1$ or $k$ defines a vertex of $\Gamma$, but $(1,k)$ is not an edge of $\Gamma$. Without loss of generality, assume that $k$ is maximal in $\mathcal{P}$ and defines a vertex of $\Gamma$. Then $k$ must be contained in exactly two edges of $\Gamma$, say edges $(i,k)$ and $(j,k)$. Without loss of generality, assume $(i,k)\in E_1$ and $(j,k)\in E_2.$ In this case, row $\mathbf{E_{i,k}}$ contributes $-1$ and row $\mathbf{E_{j,k}}$ contributes $1$ to column $\mathbf{E_{1,1} - E_{k,k}}$ of $\mathscr{L}$. As no other rows involved in $\mathscr{L}$ have nonzero entries in column $\mathbf{E_{1,1}-E_{k,k}}$, we conclude that $\mathscr{L}$ has an entry of 0 in this column.
\\*

\noindent
\textbf{Case 4}: Columns of the form $\mathbf{E_{1,1} - E_{k,k}}$ where neither $1$ nor $k$ defines a vertex in $\Gamma$. In this case, all rows involved in $\mathscr{L}$ have an entry of 0 in column $\mathbf{E_{1,1} - E_{k,k}}$.
\\*

\noindent
Thus, the claim is established. Consequently, $\det\left(\left[\widehat{B}_{\varphi}\right]\right)=0$, a contradiction. The result follows.
\end{proof}

\begin{Ex}\label{ex:cyclesproof}
Consider the poset $\mathcal{P}_1$ described in Example~\ref{ex:cycle}. The Hasse diagram of $\mathcal{P}_{Ext(\mathcal{P}_1)}$ is exactly the Hasse diagram of $\mathcal{P}_2$ in Example~\ref{ex:cycle}, with a relabeling of the vertices. Therefore, since $\mathcal{P}_2$ contains a cycle, Theorem~\ref{lem:cycle} implies that $\mathfrak{g}_A(\mathcal{P}_1)$ is not contact.
\end{Ex}

\section{Combinatorial classification }\label{sec:mainresults}
In this section we establish the first of our two main results which is the combinatorial classification of posets of height at most two which generate contact, type-A Lie poset algebras (``contact posets'').  In Section \ref{sec:01dis2}, the classification is given for height-zero, height-one, and disconnected, height-two posets (see Theorem \ref{thm:h1+dh2}). The more substantive connected, height-two case is treated in Section~\ref{sec:contwo} (see Theorem~\ref{thm:contact}).

\subsection{Height zero, height one, and disconnected, height two}\label{sec:01dis2}

\begin{theorem}\label{thm:h1+dh2}
Let $\mathcal{P}$ be a height-zero, height-one, or disconnected, height-two poset. Then $\mathfrak{g}_A(\mathcal{P})$ is contact if and only if
\begin{itemize}
    \item $|Ext(\mathcal{P}_{\{j\in\mathcal{P}~|~i\preceq j\text{ or }j\preceq i\}})|=3$ for all $i\in\mathcal{P}\backslash Ext(\mathcal{P})$ and
    \item the Hasse diagram of $\mathcal{P}_{Ext(\mathcal{P})}$ consists of two disjoint trees.
\end{itemize}
\end{theorem}
\begin{proof}
The height-zero case is immediate. To start, we prove the result for disconnected posets of height at most two. This is accomplished by showing that a disconnected poset corresponds to a contact, type-A Lie poset algebra if and only if it is a disjoint sum of two Frobenius posets. As a result of Theorem~\ref{thm:gind}, a type-A Lie poset algebra corresponding to a disconnected poset $\mathcal{P}$ has index one if and only if $\mathcal{P}$ is the disjoint sum of two Frobenius posets. Therefore, if $\mathcal{P}$ is a disconnected poset associated to a contact, type-A Lie poset algebra, then $\mathcal{P}$ is the disjoint sum of two Frobenius posets. To get the other direction, let $\mathcal{P}$ be the disjoint sum of two Frobenius posets $\mathcal{P}_1$ and $\mathcal{P}_2$. Then $$|\mathcal{P}_2|\sum_{i\in\mathcal{P}_1}E_{i,i}-|\mathcal{P}_1|\sum_{j\in\mathcal{P}_2}E_{j,j}\in Z(\mathfrak{g}_A(\mathcal{P})).$$ Thus, Theorem~\ref{lem:commute} implies that $\mathfrak{g}_A(\mathcal{P})$ is contact.
The result for disconnected posets of heights one and two now follows upon applying Theorem~\ref{thm:h2frob}. 

To establish the result, it now suffices to show that if $\mathcal{P}$ is a height-one poset for which $\mathfrak{g}_A(\mathcal{P})$ is contact, then $\mathcal{P}$ must be disconnected. Assume otherwise. Since $\mathfrak{g}_A(\mathcal{P})$ is contact, it has index equal to one. Theorem~\ref{thm:gind} implies that the Hasse diagram of $\mathcal{P}$ contains $|\mathcal{P}|$ edges and vertices. Since $\mathcal{P}$ is connected, the Hasse diagram must contain a cycle. However, as $\mathcal{P}=\mathcal{P}_{Ext(\mathcal{P})}$ for height-one posets, this is a contradiction considering Lemma~\ref{lem:cycle}. The result follows.
\end{proof}

\begin{Ex}\label{ex:h2dis}
The disconnected, height-two poset $\mathcal{P}$ with the Hasse diagram in Figure~\ref{fig:h2dis} is contact by Theorem~\ref{thm:h1+dh2}.
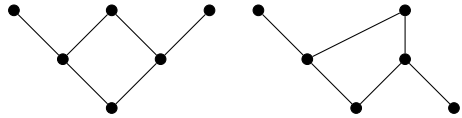
\begin{figure}[H]
$$\begin{tikzpicture}[scale=0.65]
\node (1) at (0,0) [circle, draw = black, fill = black, inner sep = 0.5mm] {};
\node (2) at (-1,1) [circle,draw=black,fill=black,inner sep=0.5mm]{};
\node (3) at (1,1)[circle,draw=black,fill=black,inner sep=0.5mm]{};
\node (4) at (-2,2)[circle,draw=black,fill=black,inner sep=0.5mm]{};
\node (5) at (0,2)[circle,draw=black,fill=black,inner sep=0.5mm]{};
\node (6) at (2,2)[circle,draw=black,fill=black,inner sep=0.5mm]{};
\node (7) at (5,0)[circle,draw=black,fill=black,inner sep=0.5mm]{};
\node (8) at (7,0)[circle,draw=black,fill=black,inner sep=0.5mm]{};
\node (9) at (4,1) [circle,draw=black,fill=black,inner sep=0.5mm]{};
\node (10) at (6,1) [circle,draw=black,fill=black,inner sep=0.5mm]{};
\node (11) at (3,2) [circle,draw=black,fill=black,inner sep=0.5mm]{};
\node (12) at (6,2)[circle,draw=black,fill=black,inner sep=0.5mm]{};
\draw (1)--(2)--(5);
\draw (2)--(4);
\draw (1)--(3)--(5);
\draw (3)--(6);
\draw (7)--(9)--(11);
\draw (9)--(12);
\draw (7)--(10)--(12);
\draw (8)--(10);
%\node[label=below:{$1$}] at (0,0){};
%\node[label=left:{$2$}] at (-1,1){};
%\node[label=right:{$3$}] at (1,1){};
%\node[label=above:{$4$}] at (-2,2){};
%\node[label=above:{$5$}] at (0,2){};
%\node[label=above:{$6$}] at (2,2){};
%\node[label=below:{$7$}] at (5,0){};
%\node[label=below:{$8$}] at (7,0){};
%\node[label=left:{$9$}] at (4,1){};
%\node[label=right:{$10$}] at (6,1){};
%\node[label=above:{$11$}] at (3,2){};
%\node[label=above:{$12$}] at (6,2){};
\end{tikzpicture}$$
    \caption{Hasse diagram of $\mathcal{P}$}
    \label{fig:h2dis}
\end{figure}
\end{Ex}

\subsection{Connected, height two}\label{sec:contwo}
To begin, we describe a recipe for the iterative construction of a sequence of posets via ``building blocks" and ``gluing rules". See the construction sequence below. We proceed by narrowing the collections of building blocks (Theorems \ref{thm:blocks} and \ref{lem:113}) and gluing rules (Theorem \ref{lem:contactsequence}) so that the limiting poset is contact. The resulting sequence is called a contact sequence. This section culminates with the characterization of contact, type-A Lie poset algebras  corresponding to connected, height-two posets as exactly those whose poset is the inductive limit of a contact sequence (Theorem \ref{thm:contact}).
\bigskip

\newpage

\begin{tcolorbox}[breakable, enhanced]
\centerline{CONSTRUCTION SEQUENCE}
\bigskip

Let $\mathcal{P}$ be a connected, height-two poset. Then each $i\in\mathcal{P}\backslash Ext(\mathcal{P})$ defines a poset $$\mathcal{P}^i=\mathcal{P}_{\{j\in\mathcal{P}~|~i\preceq j\text{ or }j\preceq i\}}$$ of the form $\mathcal{P}(n_i,1,m_i)$, where $n_i=D(\mathcal{P},i)$ and $m_i=U(\mathcal{P},i)$. Denote the set of such posets corresponding to elements of $\mathcal{P}\backslash Ext(\mathcal{P}),$ along with posets of the form $\mathcal{P}(1,1)$ corresponding to covering relations between elements of $Ext(\mathcal{P}),$ by $S$. Starting from any $\mathcal{S}_0\in S$, it is possible to form a sequence of posets $$\mathcal{S}_0=\mathcal{P}_0\subset\mathcal{P}_1\subset\hdots\subset\mathcal{P}_k=\mathcal{P},$$ where $\mathcal{P}_j$ is connected, for $j=0,\hdots,k$, and $\mathcal{P}_{j}$ is formed from $\mathcal{P}_{j-1}$ and $\mathcal{S}_{j}\in S\backslash\{\mathcal{S}_{0},\hdots,\mathcal{S}_{j-1}\}$ by identifying pairs of maximal \textup(resp., minimal\textup) elements of each, for $j=1,\hdots,k$. Such a ``gluing process" is illustrated in Example~\ref{ex:npurefrob}. In the proof of Lemma 1 in \textup{\textbf{\cite{seriesA}}}, it is shown that 
\begin{equation}\label{eqn:indinc}
\ind\mathfrak{g}_A(\mathcal{P}_0)\le \ind\mathfrak{g}_A(\mathcal{P}_1)\le\hdots\le \ind\mathfrak{g}_A(\mathcal{P}_k).
\end{equation}

\begin{Ex}\label{ex:npurefrob}
A height-two poset $\mathcal{P},$ along with the construction of $\mathcal{P}$, as outlined in the construction sequence, is illustrated in Figure~\ref{fig:Frobnonpure}.
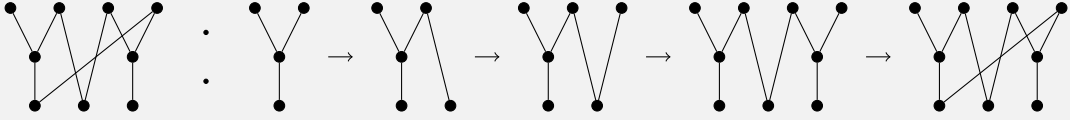
\begin{figure}[H]
$$\begin{tikzpicture}[scale=0.65]
\node [circle, draw = black, fill = black, inner sep = 0.5mm] (v1) at (-2,-0.5) {};
\node (v5) at (-1,-0.5) [circle, draw = black, fill = black, inner sep = 0.5mm] {};
\node (v3) at (-1,0.5) [circle, draw = black, fill = black, inner sep = 0.5mm] {};
\node (v2) at (-1.5,1.5) [circle, draw = black, fill = black, inner sep = 0.5mm] {};
\node [circle, draw = black, fill = black, inner sep = 0.5mm] (v4) at (-0.5,1.5) {};
\draw (v1) -- (v2) -- (v3)--(v4);
\draw (v5) -- (v3);
\node (v6) at (-2.5,1.5) [circle, draw = black, fill = black, inner sep = 0.5mm] {};
\draw (v6) -- (v1);
\node (v7) at (-3,-0.5) [circle, draw = black, fill = black, inner sep = 0.5mm] {};
\node (v8) at (-3,0.5) [circle, draw = black, fill = black, inner sep = 0.5mm] {};
\node (v9) at (-3.5,1.5) [circle, draw = black, fill = black, inner sep = 0.5mm] {};
\draw (v7) -- (v8) -- (v9);
\draw (v8) -- (v6);
\draw (v7) -- (v4);

\node [circle, draw = black, fill = black, inner sep = 0.2mm] at (0.5,1) {};
\node [circle, draw = black, fill = black, inner sep = 0.2mm] at (0.5,0) {};
\node [circle, draw = black, fill = black, inner sep = 0.5mm] (v30) at (1.5,1.5) {};
\node [circle, draw = black, fill = black, inner sep = 0.5mm] (v31) at (2.5,1.5) {};
\node [circle, draw = black, fill = black, inner sep = 0.5mm] (v29) at (2,0.5) {};
\node [circle, draw = black, fill = black, inner sep = 0.5mm] (v28) at (2,-0.5) {};
\draw (v28) -- (v29) -- (v30);
\draw (v29) -- (v31);
\draw[->] (3,0.5) -- (3.5,0.5);
\node [circle, draw = black, fill = black, inner sep = 0.5mm] (v34) at (7,1.5) {};
\node [circle, draw = black, fill = black, inner sep = 0.5mm] (v35) at (8,1.5) {};
\node [circle, draw = black, fill = black, inner sep = 0.5mm] (v33) at (7.5,0.5) {};
\node [circle, draw = black, fill = black, inner sep = 0.5mm] (v32) at (7.5,-0.5) {};
\node [circle, draw = black, fill = black, inner sep = 0.5mm] (v36) at (8.5,-0.5) {};
\node [circle, draw = black, fill = black, inner sep = 0.5mm] (v37) at (9,1.5) {};
\draw (v32) -- (v33) -- (v34) -- cycle;
\draw (v33) -- (v35) -- (v36) -- (v37);
\draw[->] (9.5,0.5) -- (10,0.5);
\node [circle, draw = black, fill = black, inner sep = 0.5mm] (v45) at (10.5,1.5) {};
\node [circle, draw = black, fill = black, inner sep = 0.5mm] (v40) at (11.5,1.5) {};
\node [circle, draw = black, fill = black, inner sep = 0.5mm] (v39) at (11,0.5) {};
\node [circle, draw = black, fill = black, inner sep = 0.5mm] (v38) at (11,-0.5) {};
\node [circle, draw = black, fill = black, inner sep = 0.5mm] (v41) at (12,-0.5) {};
\node [circle, draw = black, fill = black, inner sep = 0.5mm] (v42) at (12.5,1.5) {};
\node [circle, draw = black, fill = black, inner sep = 0.5mm] (v46) at (13.5,1.5) {};
\node [circle, draw = black, fill = black, inner sep = 0.5mm] (v43) at (13,0.5) {};
\node [circle, draw = black, fill = black, inner sep = 0.5mm] (v44) at (13,-0.5) {};
\draw (v38) -- (v39) -- (v40) -- (v41) -- (v42) -- (v43) -- (v44);
\draw (v39) -- (v45);
\draw (v43) -- (v46);
\node (v14) at (4,1.5)  [circle, draw = black, fill = black, inner sep = 0.5mm] {};
\node (v12) at (5,1.5)  [circle, draw = black, fill = black, inner sep = 0.5mm] {};
\node (v11) at (4.5,0.5)  [circle, draw = black, fill = black, inner sep = 0.5mm] {};
\node (v10) at (4.5,-0.5)  [circle, draw = black, fill = black, inner sep = 0.5mm] {};
\node (v13) at (5.5,-0.5)  [circle, draw = black, fill = black, inner sep = 0.5mm] {};
\draw (v10) -- (v11) -- (v12) -- (v13);
\draw (v11) -- (v14);
\draw[->] (6,0.5) -- (6.5,0.5);
\node [circle, draw = black, fill = black, inner sep = 0.5mm] (v15) at (15,1.5) {};
\node [circle, draw = black, fill = black, inner sep = 0.5mm] (v16) at (15.5,0.5) {};
\node [circle, draw = black, fill = black, inner sep = 0.5mm] (v18) at (16,1.5) {};
\node [circle, draw = black, fill = black, inner sep = 0.5mm] (v17) at (15.5,-0.5) {};
\node [circle, draw = black, fill = black, inner sep = 0.5mm] (v19) at (16.5,-0.5) {};
\node [circle, draw = black, fill = black, inner sep = 0.5mm] (v20) at (17,1.5) {};
\node [circle, draw = black, fill = black, inner sep = 0.5mm] (v23) at (18,1.5) {};
\node [circle, draw = black, fill = black, inner sep = 0.5mm] (v21) at (17.5,0.5) {};
\node [circle, draw = black, fill = black, inner sep = 0.5mm] (v22) at (17.5,-0.5) {};
\draw (v15) -- (v16) -- (v17);
\draw (v16) -- (v18) -- (v19) -- (v20) -- (v21) -- (v22);
\draw (v17) -- (v23) -- (v21);
\draw[->] (14,0.5) -- (14.5,0.5);
\end{tikzpicture}$$
\caption{Construction of a height-two poset}\label{fig:Frobnonpure}
\end{figure}
\end{Ex}
\end{tcolorbox}
\medskip

In the following theorem, given a connected, height-two poset $\mathcal{P}$ satisfying $\ind\mathfrak{g}_A(\mathcal{P})=1$, we determine restrictions on the form of $\mathcal{P}^i$, for $i\in\mathcal{P}\backslash Ext(\mathcal{P})$, as defined in the construction sequence.

\begin{theorem}\label{thm:blocks}
Let $\mathcal{P}$ be a connected, height-two poset such that $\ind\mathfrak{g}_A(\mathcal{P})=1$. If $i\in\mathcal{P}\backslash Ext(\mathcal{P})$, then $\mathcal{P}^i$ must be of one of the following forms: $\mathcal{P}(1,1,1)$, $\mathcal{P}(1,1,2)$, $\mathcal{P}(2,1,1)$, $\mathcal{P}(1,1,3)$, or $\mathcal{P}(3,1,1)$.
\end{theorem}
\begin{proof}
Let $i\in\mathcal{P}\backslash Ext(\mathcal{P})$. First we show that $\ind\mathfrak{g}_A(\mathcal{P}^i)\le 1$. If not, then take $\mathcal{P}_0=\mathcal{P}^i$ in the construction of $\mathcal{P}$ as outlined in the construction sequence. Considering (\ref{eqn:indinc}), we have $$1<\ind\mathfrak{g}_A(\mathcal{P}_0)\le \ind\mathfrak{g}_A(\mathcal{P}_1)\le\hdots\le \ind\mathfrak{g}_A(\mathcal{P}_k)=\ind\mathfrak{g}_A(\mathcal{P}),$$ which is a contradiction. Thus, $\ind\mathfrak{g}_A(\mathcal{P}^i)\le 1$, for all $i\in\mathcal{P}\backslash Ext(\mathcal{P})$. Now, using Theorem~\ref{thm:gind}, the result follows.
\end{proof}

Given a connected, height-two poset $\mathcal{P}$ for which $\mathfrak{g}_A(\mathcal{P})$ is contact, the next theorem further restricts the form of $\mathcal{P}^i$, for $i\in\mathcal{P}\backslash Ext(\mathcal{P})$.

\begin{theorem}\label{lem:113}
If $\mathcal{P}$ is a connected, height-two poset for which $\mathfrak{g}_A(\mathcal{P})$ is contact, then $\mathcal{P}^i\neq \mathcal{P}(1,1,3)$ or $\mathcal{P}(3,1,1)$, for $i\in\mathcal{P}\backslash Ext(\mathcal{P})$.
\end{theorem}
\begin{proof}
Let $\mathcal{P}$ be a connected, height-two poset such that $\mathfrak{g}=\mathfrak{g}_A(\mathcal{P})$ is contact with contact structure $\varphi\in\mathfrak{g}^*$. Assume that $\mathcal{P}^i=\mathcal{P}(1,1,3)$ for some $i\in\mathcal{P}\backslash Ext(\mathcal{P})$; a similar argument applies for $\mathcal{P}^i=\mathcal{P}(3,1,1)$. Without loss of generality, assume that $\mathcal{P}^i=\{1,2,3,4,5\}$ with $1\prec 2\prec 3,4,5$. Let $D_{1,p}=E_{1,1}-E_{p,p},$ fix the ordered basis $$\mathscr{B}(\mathfrak{g})=\{D_{1,p}~|~1\neq p\in\mathcal{P}\}\cup\{E_{p,q}~|~p\prec q\},$$ and consider $\left[\widehat{B}_{\varphi}\right]=\varphi\left(\widehat{C}(\mathfrak{g},\mathscr{B}(\mathfrak{g}))\right)$. The rows $\mathbf{E_{1,3}}$, $\mathbf{E_{1,4}}$, $\mathbf{E_{1,5}}$, $\mathbf{E_{2,3}}$, $\mathbf{E_{2,4}}$, and $\mathbf{E_{2,5}}$ of $\left[\widehat{B}_{\varphi}\right]$ are illustrated below with zero-columns removed.

\begin{figure}[H]
\[
  \kbordermatrix{
    & \mathbf{I} & \mathbf{D_{1,2}} & \mathbf{D_{1,3}} & \mathbf{D_{1,4}} & \mathbf{D_{1,5}} & \mathbf{E_{1,2}} \\
    \mathbf{E_{1,3}} & -\varphi(E_{1,3}) & -\varphi(E_{1,3}) & -2\varphi(E_{1,3}) & -\varphi(E_{1,3}) & -\varphi(E_{1,3}) & 0 \\
    \mathbf{E_{1,4}} & -\varphi(E_{1,4}) & -\varphi(E_{1,4}) & -\varphi(E_{1,4}) & -2\varphi(E_{1,4}) & -\varphi(E_{1,4}) & 0 \\
    \mathbf{E_{1,5}} & -\varphi(E_{1,5}) & -\varphi(E_{1,5}) & -\varphi(E_{1,5}) & -\varphi(E_{1,5}) & -2\varphi(E_{1,5}) & 0 \\
    \mathbf{E_{2,3}} & -\varphi(E_{2,3}) & \varphi(E_{2,3}) & -\varphi(E_{2,3}) & 0 & 0 & -\varphi(E_{1,3})  \\
    \mathbf{E_{2,4}} & -\varphi(E_{2,4}) & \varphi(E_{2,4}) & 0 & -\varphi(E_{2,4}) & 0 & -\varphi(E_{1,4})  \\
    \mathbf{E_{2,5}} & -\varphi(E_{2,5}) & \varphi(E_{2,5}) & 0 & 0 & -\varphi(E_{2,5}) & -\varphi(E_{1,5})  \\
  }
\]
\caption{Select rows and columns of $\left[\widehat{B}_{\varphi}\right]$}\label{fig:113mat}
\end{figure}
Consider rows $\mathbf{E_{1,3}}$, $\mathbf{E_{1,4}}$, and $\mathbf{E_{1,5}}$. Since $\varphi$ is a contact form, it follows that $$\varphi(E_{1,3}),~\varphi(E_{1,4}),~\varphi(E_{1,5})\neq 0.$$ Similarly, considering rows $\mathbf{E_{2,3}}$, $\mathbf{E_{2,4}}$, and $\mathbf{E_{2,5}}$, no two of $\varphi(E_{2,3}),~\varphi(E_{2,4}),~\varphi(E_{2,5})$ can be equal to zero; otherwise, if $\varphi(E_{2,p_1}),\varphi(E_{2,p_2})= 0$, for $p_1\neq p_2\in\{3,4,5\}$, then row $\mathbf{E_{2,p_1}}$ is in the span of row $\mathbf{E_{2,p_2}}$. Assume $\varphi(E_{2,p_1}),\varphi(E_{2,p_2})\neq 0$, for $p_1\neq p_2\in\{3,4,5\}$.

Define a linear combination of rows of $\left[\widehat{B}_{\varphi}\right]$ as follows: $$\mathscr{L}_1=\frac{1}{\varphi(E_{2,p_1})}\mathbf{E_{2,p_1}}-\frac{1}{\varphi(E_{2,p_2})}\mathbf{E_{2,p_2}}-\frac{1}{\varphi(E_{1,p_1})}\mathbf{E_{1,p_1}}+\frac{1}{\varphi(E_{1,p_2})}\mathbf{E_{1,p_2}}.$$  Evidently, $\mathscr{L}_1$ has entries of 0 in columns $\mathbf{I}$ and $\mathbf{D_{1,p}}$, for $2\le p\le 5$, and an entry of $$v_1=\frac{\varphi(E_{1,p_2})}{\varphi(E_{2,p_2})}-\frac{\varphi(E_{1,p_1})}{\varphi(E_{2,p_1})}$$ in column $\mathbf{E_{1,2}}$. Since $\varphi$ is a contact form on $\mathfrak{g},$ $v_1\neq 0$. It follows that $\varphi(E_{2,p_3})\neq 0,$ for $p_3\in\{3,4,5\}\backslash\{p_1,p_2\}$; otherwise, row $\mathbf{E_{2,p_3}}$ would be in the span of rows $\mathbf{E_{1,p_1}}$, $\mathbf{E_{1,p_2}}$, $\mathbf{E_{2,p_1}}$, and $\mathbf{E_{2,p_2}}$, contradicting the assumption that $\varphi$ is a contact form on $\mathfrak{g}$. 

Now, define a second linear combination of rows of $\left[\widehat{B}_{\varphi}\right]$ as follows: $$\mathscr{L}_2=\frac{1}{\varphi(E_{2,p_3})}\mathbf{E_{2,p_3}}-\frac{1}{\varphi(E_{2,p_2})}\mathbf{E_{2,p_2}}-\frac{1}{\varphi(E_{1,p_3})}\mathbf{E_{1,p_3}}+\frac{1}{\varphi(E_{1,p_2})}\mathbf{E_{1,p_2}}.$$ As in the case of $\mathscr{L}_1$, $\mathscr{L}_2$ has a single non-zero entry, say $v_2$, in column $\mathbf{E_{1,2}}$. Therefore, $\frac{1}{v_1}\mathscr{L}_1-\frac{1}{v_2}\mathscr{L}_2$ is equal to the zero row and provides a nontrivial dependence relation between the rows of $\left[\widehat{B}_{\varphi}\right]$; that is, $\det\left(\left[\widehat{B}_{\varphi}\right]\right)= 0$, a contradiction. The result follows.
\end{proof}

Thus, to construct a connected, height-two poset $\mathcal{P}$ for which $\mathfrak{g}_A(\mathcal{P})$ is contact, the building blocks are necessarily of the form $\mathcal{P}(1,1)$, $\mathcal{P}(1,1,1)$, $\mathcal{P}(1,1,2)$, or $\mathcal{P}(2,1,1)$. Next, we consider how such posets must be combined so that the resulting type-A Lie poset algebra is contact. As an intermediate step, we determine how such posets must be combined so that the resulting type-A Lie poset algebra has index one.

Let $\mathcal{S}$ be a poset of the form $\mathcal{P}(1,1)$, $\mathcal{P}(1,1,1)$, $\mathcal{P}(2,1,1)$, or $\mathcal{P}(1,1,2)$ and $\mathcal{Q}$ be a connected, height-two poset. We list all ways of ``gluing" the posets $\mathcal{S}$ and $\mathcal{Q}$ by identifying minimal (resp., maximal) elements of $\mathcal{S}$ with distinct minimal (resp., maximal) elements of $\mathcal{Q}$. If $\mathcal{S}$ is of the form $\mathcal{P}(1,1)$ or $\mathcal{P}(1,1,1)$, then $Ext(\mathcal{S})=\{a_1,c\}$ with $c\prec_{\mathcal{S}}a_1$;  and if $\mathcal{S}$ is of the form $\mathcal{P}(2,1,1)$ or $\mathcal{P}(1,1,2)$, then $Ext(\mathcal{S})=\{a_1,a_2,c\}$ with either $c\prec_{\mathcal{S}} a_1,a_2$ or $a_1,a_2\prec_{\mathcal{S}} c$. Further, assume $x,y,z\in Ext(\mathcal{Q})$; ongoing, we assume that if $|Ext(\mathcal{Q})|=2$, then $x,y\in Ext(\mathcal{Q}),$ and any rules defined below involving $z$ do not apply. Since the gluing rules are defined by identifying minimal elements and maximal elements of $\mathcal{S}$ and $\mathcal{Q}$, assume that if $c,a_1$, or $a_2$ are identified with elements of $\mathcal{Q}$, then those elements are $x,y,$ or $z$, respectively. To ease notation, let $\sim_{\mathcal{P}}$ denote that two elements of a poset $\mathcal{P}$ are related, and let $\nsim_{\mathcal{P}}$ denote that two elements are not related; that is, for $i,j\in\mathcal{P}$, $i\sim_{\mathcal{P}} j$ denotes that $i\preceq_{\mathcal{P}}j$ or $j\preceq_{\mathcal{P}}i$, and $i\nsim_{\mathcal{P}} j$ denotes that both $i\npreceq_{\mathcal{P}}j$ and $j\npreceq_{\mathcal{P}}i$. The following Table~\ref{tab:h2fassem} lists all possible ways of identifying the elements $c,a_1,a_2\in\mathcal{S}$ with the elements $x,y,z\in\mathcal{Q}$. The last column of Table~\ref{tab:h2fassem} records the attendant contributions to the index; that is, if $\mathcal{P}$ is the poset resulting from gluing $\mathcal{S}$ to $\mathcal{Q}$, then this column gives $\ind\mathfrak{g}_A(\mathcal{P})-\ind\mathfrak{g}_A(\mathcal{Q})$.

\begin{theorem}\label{lem:table}
The table below summarizes the contribution to the index of a height-two poset upon adjoining a copy of $\mathcal{P}(1,1)$, $\mathcal{P}(1,1,1)$, $\mathcal{P}(1,1,2)$, or $\mathcal{P}(2,1,1)$ as described above.

\begin{table}[H]
\centering
\begin{tabular}{c|c|c|c|c}
Gluing Rule & $c$                       & $a_1$                         & $a_2$                         & Contribution to the Index \\ \hline
$A_1$ & $c\neq x$ & $a_1=y$         & $a_2\neq z$ & $\ind\mathfrak{g}_A(\mathcal{S})$                         \\
$A_2$ & $c\neq x$ & $a_1\neq y$             & $a_2=z$ & $\ind\mathfrak{g}_A(\mathcal{S})$                         \\
$B$ & $c\neq x$ & $a_1=y$       & $a_2=z$                 & $\ind\mathfrak{g}_A(\mathcal{S})+1$                         \\
$C$ & $c=x$ & $a_1\neq y$                 & $a_2\neq z$                 & $\ind\mathfrak{g}_A(\mathcal{S})$                         \\
$D_1$ & $c=x$                 & $a_1=y$, $y\sim x$ & $a_2\neq z$  & $\ind\mathfrak{g}_A(\mathcal{S})$                         \\
$D_2$ & $c=x$                 & $a_1\neq y$  & $a_2= z$, $z\sim x$  & $\ind\mathfrak{g}_A(\mathcal{S})$                         \\
$E_1$ & $c=x$                 & $a_1= y$, $y\nsim x$  & $a_2\neq z$  & $\ind\mathfrak{g}_A(\mathcal{S})+1$                         \\
$E_2$ & $c=x$                 & $a_1\neq y$  & $a_2= z$, $z\nsim x$  & $\ind\mathfrak{g}_A(\mathcal{S})+1$                         \\
$F$ & $c=x$                 & $a_1= y$, $y\sim x$  & $a_2= z$, $z\sim x$  & $\ind\mathfrak{g}_A(\mathcal{S})$                         \\
$G_1$ & $c=x$                 & $a_1= y$, $y\sim x$  & $a_2= z$, $z\nsim x$  & $\ind\mathfrak{g}_A(\mathcal{S})+1$                         \\
$G_2$ & $c=x$                 & $a_1= y$, $y\nsim x$  & $a_2= z$, $z\sim x$  & $\ind\mathfrak{g}_A(\mathcal{S})+1$                         \\
$H$ & $c=x$                 & $a_1= y$, $y\nsim x$  & $a_2= z$, $z\nsim x$  & $\ind\mathfrak{g}_A(\mathcal{S})+2$                         \\
\end{tabular}
\caption{Height-two gluing rules}\label{tab:h2fassem}
\end{table}

\end{theorem}
\begin{proof}
Apply Theorem~\ref{thm:gind}.
\end{proof}

\begin{remark}
The only rules that apply to adjoining a copy of $\mathcal{P}(1,1)$ or $\mathcal{P}(1,1,1)$ in Table~\ref{tab:h2fassem} are $A_1$, $C$, $D_1$, and $E_1$. Note that applying $D_1$ to adjoin a copy of $\mathcal{P}(1,1)$ does not result in a new poset.
\end{remark}

If $\mathcal{P}$ is a connected, height-two poset such that $\mathfrak{g}_A(\mathcal{P})$ is a contact, type-A Lie poset algebra, then not all of the gluing rules given in Table~\ref{tab:h2fassem} can be used to construct $\mathcal{P}.$ Furthermore, one particular building block must appear exactly once in the construction sequence of $\mathcal{P}.$ These restrictions are detailed in the following theorem.

\begin{theorem}\label{lem:contactsequence}
Let $\{\mathcal{P}_i\}_{i=0}^n$ be a sequence of connected, height-two posets such that $\mathcal{P}_0=\mathcal{S}_0=\mathcal{P}(1,1)$, $\mathcal{P}(1,1,1)$, $\mathcal{P}(1,1,2)$, or $\mathcal{P}(2,1,1)$ and $\mathcal{P}_j$ is obtained by adjoining a copy of $\mathcal{S}_j=\mathcal{P}(1,1)$, $\mathcal{P}(1,1,1)$, $\mathcal{P}(1,1,2)$, or $\mathcal{P}(2,1,1)$ to $\mathcal{P}_{j-1}$, for $j=1,...,n$, by applying a gluing rule from the set $$\{A_1,A_2,B,C, D_1,D_2,E_1,E_2,F,G_1,G_2,H\}.$$ If $\mathfrak{g}_A(\mathcal{P}_n)$ is contact, then 
\begin{enumerate}
    \item $\mathcal{P}_j$ is obtained by adjoining $\mathcal{S}_j$ to $\mathcal{P}_{j-1}$, for $j=1,...,n$, by applying a gluing rule from the set $$\{A_1,A_2,C, D_1,D_2,F\};\text{ and}$$
    \item $\mathcal{S}_j=\mathcal{P}(1,1,1)$, for exactly one value of $j\in\{0,1,...,n\}$.
\end{enumerate}
\end{theorem}
\begin{proof}
First we prove 1. Let $\mathcal{Q}$ be a connected, height-two poset and let $\mathcal{P}$ be a poset formed by adjoining $\mathcal{S}=\mathcal{P}(1,1)$, $\mathcal{P}(1,1,1)$, $\mathcal{P}(1,1,2)$, or $\mathcal{P}(2,1,1)$ to $\mathcal{Q}$ by applying a rule from the set $\{B, E_1,E_2,G_1,G_2,H\}$. Let $x,y,z\in\mathcal{Q}$ and $c,a_1,a_2\in\mathcal{S}$ be as in Theorem~\ref{lem:table}. Since $\mathcal{Q}$ is connected, so is $\mathcal{Q}_{Ext(\mathcal{Q})}$. Thus, there exists a path $P_{yz}$ from $y$ to $z$, a path $P_{xy}$ from $x$ to $y$, and a path $P_{xz}$ from $x$ to $z$ in the Hasse diagram of $\mathcal{Q}_{Ext(\mathcal{Q})}$. Note that the Hasse diagram of $\mathcal{P}_{Ext(\mathcal{P})}$ contains the Hasse diagram of $\mathcal{Q}_{Ext(\mathcal{Q})}$ as a subgraph. Three cases arise.
\\*

\noindent
\textbf{Case 1}: $\mathcal{P}$ is formed by adjoining $\mathcal{S}$ to $\mathcal{Q}$ using rule $B$. Combining the edges of $P_{yz}$ with the edges $\{y,c\}$ and $\{c,z\}$ yields a cycle in the Hasse diagram of $\mathcal{P}_{Ext(\mathcal{P})}$.
\\*

\noindent
\textbf{Case 2}: $\mathcal{P}$ is formed by adjoining $\mathcal{S}$ to $\mathcal{Q}$ using rule $E_1$, $G_2$, or $H$. Combining the edges of $P_{xy}$ with the edge $\{x,y\}$ yields a cycle in the Hasse diagram of $\mathcal{P}_{Ext(\mathcal{P})}$.
\\*

\noindent
\textbf{Case 3}: $\mathcal{P}$ is formed by adjoining $\mathcal{S}$ to $\mathcal{Q}$ using rule $E_2$ or $G_1$. Combining the edges of $P_{xz}$ with the edge $\{x,z\}$ yields a cycle in the Hasse diagram of $\mathcal{P}_{Ext(\mathcal{P})}$.
\\*

\noindent
In each of the three cases, there is a cycle in the Hasse diagram of $\mathcal{P}_{Ext(\mathcal{P})}$, so $\mathfrak{g}_A(\mathcal{P})$ is not contact by Theorem~\ref{lem:cycle}.

For 2, we first show that there is at most one $j\in\{0,1,\hdots,n\}$ for which $\mathcal{S}_j=\mathcal{P}(1,1,1)$. Using Theorem~\ref{lem:table}, note that if $m$ is the number of values of $j\in\{0,1,\hdots,n\}$ for which $\mathcal{S}_j=\mathcal{P}(1,1,1)$, then $m\le \ind\mathfrak{g}_A(\mathcal{P}_n)=1$. Now, to see that the number of such $j\in\{0,1,\hdots,n\}$ is at least 1, consider part 1 in conjunction with Theorem~\ref{lem:table}. If no $\mathcal{S}_j=\mathcal{P}(1,1,1)$, for $j\in\{0,\hdots,n\}$, then $\ind\mathfrak{g}_A(\mathcal{P})=0$. Thus, since contact Lie algebras have index 1, the result follows.
\end{proof}

We now have the tools necessary to define a contact sequence of posets. The following definitions will facilitate, and be followed by, the statement and proof of the first main result of this paper (Theorem~\ref{thm:contact}).

\begin{definition}
Let $\{\mathcal{P}_i\}_{i=0}^n$ be a sequence of connected, height-two posets such that $\mathcal{P}_0=\mathcal{S}_0=\mathcal{P}(1,1)$, $\mathcal{P}(1,1,1)$, $\mathcal{P}(1,1,2)$, or $\mathcal{P}(2,1,1)$ and $\mathcal{P}_j$ is obtained by adjoining a copy of $\mathcal{S}_j=\mathcal{P}(1,1)$, $\mathcal{P}(1,1,1)$, $\mathcal{P}(1,1,2)$, or $\mathcal{P}(2,1,1)$ to $\mathcal{P}_{j-1}$, for $j=1,...,n$, by applying a gluing rule from the set $\{A_1,A_2,C, D_1,D_2,F\}$. If there is a unique $j\in \{0,1,\hdots,n\}$ for which $\mathcal{S}_j=\mathcal{P}(1,1,1)$, then we refer to the sequence as a contact sequence.
\end{definition}

The following remark establishes conventions which will be used ongoing.

\begin{remark}
Given a contact sequence $\{\mathcal{P}_i\}_{i=0}^n$, assume that $\mathcal{P}_0=\mathcal{P}(1,1,1)$, where $\mathcal{P}_0=\{1,2,3\}$ with $1\prec 2\prec 3$. Also assume that, for $j=1,\hdots,n$, if $\mathcal{P}_j$ is obtained by adjoining $\mathcal{S}_j$ to $\mathcal{P}_{j-1}$ then
\begin{itemize}
    \item if $\mathcal{S}_j=\mathcal{P}(1,1)$, then $\mathcal{S}_j=\{x_j,y_j\}\subset \mathcal{P}_j$ with $x_j\prec y_j$;
    \item if $\mathcal{S}_j=\mathcal{P}(2,1,1)$, then $\mathcal{S}_j=\{y_j,z_j,m_j,x_j\}\subset\mathcal{P}_j$ with $y_j,z_j\prec m_j\prec x_j$; and
    \item if $\mathcal{S}_j=\mathcal{P}(1,1,2)$,  then $\mathcal{S}_j=\{x_j,m_j,y_j,z_j\}\subset\mathcal{P}_j$ with $x_j\prec m_j\prec y_j,z_j$.
\end{itemize}
Note that some elements of $\mathcal{P}_n$ receive multiple labels.
\end{remark}

\begin{definition}\label{def:contactfun}
Let $\{\mathcal{P}_i\}_{i=0}^n$ be a contact sequence. Define the functional $\varphi_{\mathcal{P}_j}\in(\mathfrak{g}_A(\mathcal{P}_j))^*$, for $j=0,\dots,n,$ recursively as follows:
\bigskip

\noindent
\textbf{\textup{Step 0}}: $\varphi_{\mathcal{P}_0}=E^*_{2,2}+E^*_{1,3}+E^*_{2,3}$
\bigskip

\noindent
\textbf{\textup{Step j}}:  If $\mathcal{P}_{j}$ is formed from $\mathcal{P}_{j-1}$ and $\mathcal{S}_j$, for $j=1,\dots,n$, by applying rule
    \begin{itemize}
        \item $\textup{A}_1,\textup{A}_2$, or $\textup{C}$, then $$\varphi_{\mathcal{P}_j}=\begin{cases}
        \varphi_{\mathcal{P}_{j-1}}+E^*_{x_j,y_j}, & \mathcal{S}_j=\mathcal{P}(1,1); \\
        \\
        \varphi_{\mathcal{P}_{j-1}}+E^*_{y_j,x_j}+E^*_{z_j,x_j}+E^*_{z_j,m_j}, & \mathcal{S}_j=\mathcal{P}(2,1,1); \\
        \\
        \varphi_{\mathcal{P}_{j-1}}+E^*_{x_j,y_j}+E^*_{x_j,z_j}+E^*_{m_j,z_j}, & \mathcal{S}_j=\mathcal{P}(1,1,2). \\
        \end{cases}.$$
        \item $\textup{D}_1$, then
        \[\varphi_{\mathcal{P}_j} =  \begin{cases} 
      \varphi_{\mathcal{P}_{j-1}}+E^*_{x_j,z_j}+E^*_{m_j,z_j}, & \mathcal{S}_j=\mathcal{P}(1,1,2); \\
                                                &                        \\
       \varphi_{\mathcal{P}_{j-1}}+E^*_{z_j,x_j}+E^*_{z_j,m_j}, & \mathcal{S}_j=\mathcal{P}(2,1,1).
   \end{cases}
\]
        \item $\textup{D}_2$, then
        \[\varphi_{\mathcal{P}_j} =  \begin{cases} 
      \varphi_{\mathcal{P}_{j-1}}+E^*_{x_j,y_j}+E^*_{m_j,z_j}, & \mathcal{S}_j=\mathcal{P}(1,1,2); \\
                                                &                        \\
       \varphi_{\mathcal{P}_{j-1}}+E^*_{y_j,x_j}+E^*_{z_j,m_j}, & \mathcal{S}_j=\mathcal{P}(2,1,1).
   \end{cases}
\]
        \item $\textup{F}$, then
        \[\varphi_{\mathcal{P}_j} =  \begin{cases} 
      \varphi_{\mathcal{P}_{j-1}}+E^*_{m_j,z_j}, & \mathcal{S}_j=\mathcal{P}(1,1,2); \\
                                                &                        \\
       \varphi_{\mathcal{P}_{j-1}}+E^*_{z_j,m_j}, & \mathcal{S}_j=\mathcal{P}(2,1,1).
   \end{cases}
\]
    \end{itemize}
\end{definition}

%We are now in a position to state the main theorem of this subsection.

\begin{theorem}\label{thm:contact}
Let $\mathcal{P}$ be a connected, height-two poset. Then $\mathfrak{g}_A(\mathcal{P})$ is contact if and only if there exists a contact sequence $\{\mathcal{P}_i\}_{i=0}^n$ such that $\mathcal{P}_n=\mathcal{P}$.
\end{theorem}

The following is an immediate corollary of Theorem~\ref{thm:contact}.

\begin{theorem}
Let $\mathcal{P}$ be a connected, height-two poset. Then $\mathfrak{g}_A(\mathcal{P})$ is contact if and only if 
\begin{itemize}
    \item $|Ext(\mathcal{P}_{\{j\in\mathcal{P}~|~i\preceq j\text{ or }j\preceq i\}})|=2$ or $3$ for all $i\in\mathcal{P}\backslash Ext(\mathcal{P})$;
    \item there exists a unique $i\in\mathcal{P}\backslash Ext(\mathcal{P})$ satisfying $|Ext(\mathcal{P}_{\{j\in\mathcal{P}~|~i\preceq j\text{ or }j\preceq i\}})|=2$; and
    \item the Hasse diagram of $\mathcal{P}_{Ext(\mathcal{P})}$ is a tree.
\end{itemize}
\end{theorem}

\subsubsection{Proof of Theorem~\ref{thm:contact}}

It follows from Theorem~\ref{lem:contactsequence} that if $\mathfrak{g}_A(\mathcal{P})$ is contact with $\mathcal{P}$ a connected, height-two poset, then any sequence of posets corresponding to $\mathcal{P}$ as described in the construction sequence must form a contact sequence $\{\mathcal{P}_i\}_{i=0}^n$ such that $\mathcal{P}_n=\mathcal{P}$. Thus, to establish Theorem~\ref{thm:contact}, we must show that for any contact sequence $\{\mathcal{P}_i\}_{i=0}^n$, if $\mathcal{P}_n=\mathcal{P}$, then $\mathfrak{g}_A(\mathcal{P})$ is contact. To do this, we show that $\varphi_{\mathcal{P}}$, as defined in Definition~\ref{def:contactfun}, is a contact structure on $\mathfrak{g}_A(\mathcal{P})$. 

We first establish that $\varphi_{\mathcal{P}}$ is a regular one-form on $\mathfrak{g}_A(\mathcal{P})$ (see Theorem \ref{thm:regular}). We require the following preliminary lemma.

\begin{lemma}\label{lem:extrel}
Let $\varphi_{\mathcal{P}}=\sum c_{i,j}E^*_{i,j}\in (\mathfrak{g}_A(\mathcal{P}))^*$ be the functional of Definition~\ref{def:contactfun}. If $L\in \text{ker }(B_{\varphi_{\mathcal{P}}})$, then $E_{i,j}^*(L)=0$, for all $i\neq j\in \mathcal{P}$ such that $c_{i,j}\neq 0$.
\end{lemma}
\begin{proof}
To establish the result, we provide a procedure which allows one to show that $E_{i,j}^*(L)=0$, for all $i\neq j\in \mathcal{P}$ such that $c_{i,j}\neq 0$, in $|Rel_E(\mathcal{P})|+|\mathcal{P}\backslash Ext(\mathcal{P})|$ steps. 

Define the graph $T_0$ with vertex set $\{p~|~p\in\mathcal{P}\}$ and edge set $\{\{p,q\}~|~c_{p,q}\neq 0\}$. Such a graph $T_0$ has been called the ``directed graph" of the one-form $\varphi_{\mathcal{P}}$ \textup(see \textbf{\cite{G2}}\textup), and it shall serve as a book-keeping device in the procedure that follows. In particular, once we show that $E_{i,j}^*(L)=0,$ we remove edge $\{i,j\}$ from $T_0,$ yielding a new graph $T_1,$ and so on.

We claim that $T_0$ is a tree. To see this, first note that the subgraph of $T_0$ corresponding to the elements of $Ext(\mathcal{P})$ is isomorphic to the Hasse diagram of $\mathcal{P}_{{Ext(\mathcal{P})}}$, which is a tree by Theorem~\ref{lem:cycle}. Now, by the definition of $\varphi_{\mathcal{P}}$, for each element $p\in \mathcal{P}\backslash Ext(\mathcal{P})$ there exists a unique $q\in \mathcal{P}$ -- in particular, $q\in Ext(\mathcal{P})$ -- such that $c_{p,q}$ or $c_{q,p}\neq 0$. Thus, the claim is established. 

The procedure is outlined in the following steps.
\\*

\noindent
\textbf{Step 1:} Since $T_0$ is a tree, there exists a free vertex. Without loss of generality, assume that the free vertex corresponds to $p_0\in \mathcal{P}$ with $p_0\prec q_0$ and $c_{p_0,q_0}\neq 0$. Then
$$\varphi_{\mathcal{P}}\bigg(\bigg[E_{p_0,p_0}-\frac{1}{|\mathcal{P}|}\sum_{p\in\mathcal{P}}E_{p,p},L\bigg]\bigg)=\pm E^*_{p_0,q_0}(L)=0.$$ Remove the edge $\{p_0, q_0\}$ in $T_0$, resulting in a tree with one less edge, denoted $T_1$.
\\*

\noindent
\textbf{Step $\mathbf{m}$:} Consider the graph $T_{m-1}$ formed in Step $m-1$. Since $T_{m-1}$ is a tree, there exists a free vertex. Without loss of generality, assume that the free vertex corresponds to $p_m\in \mathcal{P}$ with $p_m\prec q_m$ and $c_{p_m,q_m}\neq 0$. Then, taking into account the entries of $L$ already shown to be equal to zero in steps 1 through $m-1$,
$$\varphi_{\mathcal{P}}\bigg(\bigg[E_{p_m,p_m}-\frac{1}{|\mathcal{P}|}\sum_{p\in\mathcal{P}}E_{p,p},L\bigg]\bigg)=\pm E^*_{p_m,q_m}(L)=0.$$ Remove the edge $\{p_m, q_m\}$ in $T_{m-1}$, resulting in a tree with one less edge, denoted $T_m$.
\\*

\noindent
Since $T_0$ is a tree consisting of $|Rel_E(\mathcal{P})|+|\mathcal{P}\backslash Ext(\mathcal{P})|$ edges, the above procedure terminates with an empty graph in $|Rel_E(\mathcal{P})|+|\mathcal{P}\backslash Ext(\mathcal{P})|$ steps. The result follows.
\end{proof}

\begin{theorem}\label{thm:regular}
If $\{\mathcal{P}_i\}_{i=0}^n$ is a contact sequence, $\mathcal{P}=\mathcal{P}_n$, and $\mathfrak{g}=\mathfrak{g}_A(\mathcal{P})$, then the functional $\varphi_{\mathcal{P}}\in\mathfrak{g}^*$ given in Definition~\ref{def:contactfun} is a regular one-form on $\mathfrak{g}$; that is, $\dim\ker(B_{\varphi_{\mathcal{P}}})=1.$
\end{theorem}
\begin{proof}
Let $\varphi_{\mathcal{P}}=\sum c_{i,j}E^*_{i,j}\in \mathfrak{g}^*$ be the functional of Definition~\ref{def:contactfun} and $L\in \ker(B_{\varphi_{\mathcal{P}}})$. By Lemma~\ref{lem:extrel},  $E_{i,j}^*(L)=0$, for all $i\neq j\in \mathcal{P}$ such that $c_{i,j}\neq 0$. To completely determine the form of $L$, we will consider the restrictions placed on the entries of $L$ by the relations $\varphi_{\mathcal{P}}([x, L])=0$, for $$x\in\mathscr{B}(\mathfrak{g})=\{E_{1,1}-E_{p,p}~|~1<p\le |P|\}\cup\{E_{p,q}~|~p,q\in\mathcal{P},p\prec q\}.$$
To start, note that for $x\in \{E_{1,1}-E_{p,p}~|~1<p\le |P|\}$, one has 
\begin{equation}\label{eq:sum}
\varphi_{\mathcal{P}}([x, L])=\sum d_{i,j}E^*_{i,j}(L),
\end{equation}
for $d_{i,j}\in\mathbf{k}$, where the sum is over pairs $(i,j)$ for which $c_{i,j}\neq 0$. Considering Lemma~\ref{lem:extrel}, all such entries of $L$ occurring in (\ref{eq:sum}) must be equal to 0.

Relations of the form $\varphi_{\mathcal{P}}([x, L])=0$, for $x\in \{E_{p,q}~|~p,q\in\mathcal{P},p\prec q\}$, we break into four groups.
\\*

\noindent
\textbf{Group 1:} $E_{p,q}$, for $p,q\in \{1,2,3\}$.
\begin{itemize}
    \item $\varphi_{\mathcal{P}}([E_{1,2}, L])=E^*_{2,3}(L)=0$;
    \item $\varphi_{\mathcal{P}}([E_{2,3}, L])=-E^*_{1,2}(L)+E^*_{3,3}(L)-E^*_{2,2}(L)=0$;
    \item $\varphi_{\mathcal{P}}([E_{1,3}, L])=E^*_{3,3}(L)-E^*_{1,1}(L)=0$.
\end{itemize}
The conditions above imply that $E^*_{2,3}(L)=0$ and $$E^*_{1,1}(L)=E^*_{3,3}(L)=E^*_{1,2}(L)+E^*_{2,2}(L).$$
\medskip

\noindent
\textbf{Group 2:} $E_{p,q}$, for $p,q\in \{x_i,y_i\}$ such that $\mathcal{P}_i$ is obtained from $\mathcal{P}_{i-1}$ by adjoining $\mathcal{S}_i=\{x_i,y_i\}$ with $x_i\prec y_i$.
\begin{itemize}
    \item $B_{\varphi_{\mathcal{P}}}([E_{x_i,y_i}, L])=E^*_{y_i,y_i}(L)-E^*_{x_i,x_i}(L)=0$.
\end{itemize}
The conditions above imply that $$E^*_{x_i,x_i}(L)=E^*_{y_i,y_i}(L).$$
\medskip

\noindent
\textbf{Group 3:} $E_{p,q}$, for $p,q\in \{x_i,m_i,y_i,z_i\}$ such that $\mathcal{P}_i$ is obtained from $\mathcal{P}_{i-1}$ by adjoining $\mathcal{S}_i=\{x_i,m_i,y_i,z_i\}$ with $y_i,z_i\prec m_i\prec x_i$.
\begin{itemize}
    \item $B_{\varphi_{\mathcal{P}}}([E_{y_i,m_i}, L])=E^*_{m_i,x_i}(L)=0$;
    \item $B_{\varphi_{\mathcal{P}}}([E_{y_i,x_i}, L])=E^*_{x_i,x_i}(L)-E^*_{y_i,y_i}(L)=0$;
    \item $B_{\varphi_{\mathcal{P}}}([E_{z_i,m_i}, L])=E^*_{m_i,x_i}(L)+E^*_{m_i,m_i}(L)-E^*_{z_i,z_i}(L)=0$;
    \item $B_{\varphi_{\mathcal{P}}}([E_{z_i,x_i}, L])=E^*_{x_i,x_i}(L)-E^*_{z_i,z_i}(L)=0$;
    \item $B_{\varphi_{\mathcal{P}}}([E_{m_i,x_i}, L])=-E^*_{y_i,m_i}(L)-E^*_{z_i,m_i}(L)=0$.
\end{itemize}
The conditions above, along with Lemma~\ref{lem:extrel}, imply that $$E^*_{y_i,m_i}(L)=E^*_{z_i,m_i}(L)=E^*_{m_i,x_i}(L)=0$$ and $$E^*_{x_i,x_i}(L)=E^*_{m_i,m_i}(L)=E^*_{y_i,y_i}(L)=E^*_{z_i,z_i}(L).$$
\medskip

\noindent
\textbf{Group 4:} $E_{p,q}$, for $p,q\in \{x_i,m_i,y_i,z_i\}$ such that $\mathcal{P}_i$ is obtained from $\mathcal{P}_{i-1}$ by adjoining $\mathcal{S}_i=\{x_i,m_i,y_i,z_i\}$ with $x_i\prec m_i\prec y_i,z_i$.
\begin{itemize}
    \item $B_{\varphi_{\mathcal{P}}}([E_{x_i,m_i}, L])=E^*_{m_i,y_i}(L)+E^*_{m_i,z_i}(L)=0$;
    \item $B_{\varphi_{\mathcal{P}}}([E_{x_i,y_i}, L])=E^*_{y_i,y_i}(L)-E^*_{x_i,x_i}(L)=0$;
    \item $B_{\varphi_{\mathcal{P}}}([E_{x_i,z_i}, L])=E^*_{z_i,z_i}(L)-E^*_{x_i,x_i}(L)=0$;
    \item $B_{\varphi_{\mathcal{P}}}([E_{m_i,y_i}, L])=-E^*_{x_i,m_i}(L)=0$;
    \item $B_{\varphi_{\mathcal{P}}}([E_{m_i,z_i}, L])=-E^*_{x_i,m_i}(L)+E^*_{z_i,z_i}(L)-E^*_{m_i,m_i}(L)=0$.
\end{itemize}
The conditions above, along with Lemma~\ref{lem:extrel}, imply $$E^*_{x_i,m_i}(L)=E^*_{m_i,y_i}(L)=E^*_{m_i,z_i}(L)=0$$ and $$E^*_{x_i,x_i}(L)=E^*_{m_i,m_i}(L)=E^*_{y_i,y_i}(L)=E^*_{z_i,z_i}(L).$$
\medskip

\noindent
Hence, as a result of the conditions found on the entries of $L$ above, the fact that $L\in\mathfrak{sl}(|\mathcal{P}|)$, and the connectedness of $\mathcal{P}$, we find that $$\ker(B_{\varphi_{\mathcal{P}}})=\text{span}\left\{\sum_{2\neq p\in\mathcal{P}}E_{p,p}+(1-|\mathcal{P}|)E_{2,2}+|P|E_{1,2}\right\}.$$ Thus, since $\dim \ker(B_{\varphi_{\mathcal{P}}})=1,$ $\varphi_{\mathcal{P}}$ is regular.
\end{proof}

Now, to establish Theorem~\ref{thm:contact}, we show that $\varphi_{\mathcal{P}}$ is a contact form on $\mathfrak{g}=\mathfrak{g}_A(\mathcal{P})$. Take $\mathscr{B}(\mathfrak{g})$ to be a basis for $\mathfrak{g}$ which contains the element $$L=\sum_{2\neq p\in\mathcal{P}}E_{p,p}+(1-|\mathcal{P}|)E_{2,2}+|P|E_{1,2}\in \ker(B_{\varphi_{\mathcal{P}}}).$$ Set $\left[B_{\varphi_{\mathcal{P}}}\right]=\varphi_{\mathcal{P}}\left(C(\mathfrak{g},\mathscr{B}(\mathfrak{g}))\right)$ and $\left[\widehat{B}_{\varphi_{\mathcal{P}}}\right]=\varphi_{\mathcal{P}}\left(\widehat{C}(\mathfrak{g},\mathscr{B}(\mathfrak{g}))\right)$. Note that $\left[B_{\varphi_{\mathcal{P}}}\right]$ has rank $\dim\mathfrak{g}-1$, since $\ind\mathfrak{g}=1$, and has a single zero row and column corresponding to $L\in \ker(B_{\varphi_{\mathcal{P}}})$; denote by $\left[B'_{\varphi_{\mathcal{P}}}\right]$ the submatrix of full rank obtained from $\left[B_{\varphi_{\mathcal{P}}}\right]$ by removing the zero row and column corresponding to $L$. Now, computing the determinant of $\left[\widehat{B}_{\varphi_{\mathcal{P}}}\right]$ by expanding on row $\mathbf{L}$ followed by column $\mathbf{L}$, we have 
$$\det \left(\left[\widehat{B}_{\varphi_{\mathcal{P}}}\right]\right)=\varphi_{\mathcal{P}}(L)^2\det\left(\left[B'_{\varphi_{\mathcal{P}}}\right]\right)=(1-|\mathcal{P}|)^2\det\left(\left[B'_{\varphi_{\mathcal{P}}}\right]\right)\neq 0.$$ Therefore, $\mathfrak{g}$ is contact with contact form $\varphi_{\mathcal{P}}$ by Theorem~\ref{thm:det}. The result follows.\qed

\section{Rigidity}\label{sec:rigid}

In this section, we prove the rigidity result noted in the introduction (see Theorem \ref{thm:main2}). The proof depends on the following result of Coll and Gerstenhaber, which itself is a corollary to their more general theorem regarding Lie semi-direct products, for which type-A Lie poset algebras are the prime example.  To set the notation, let $\mathfrak{g}_A(\mathcal{P})$ be as above, $\mathfrak{h}$ be the standard Cartan subalgebra of $\mathfrak{g}_A(\mathcal{P})$ with linear dual $\mathfrak{h}^*$, $\mathfrak{c}=Z(\mathfrak{g}_A(\mathcal{P}))$, and the $H^i$'s designate cohomology classes of Chevalley-Eilenberg or simplicial type, depending on whether the first argument is a Lie algebra or a simplicial complex, respectively.

\begin{theorem}[Coll and Gerstenhaber \textbf{\cite{CG}}, 2017] \label{CG}
\[H^2(\mathfrak{g}_A(\mathcal{P}),\mathfrak{g}_A(\mathcal{P}))=\left(\bigwedge\nolimits^{\!2}\mathfrak{h}^{*}\bigotimes\mathfrak{c}\right)\quad\bigoplus\quad\left(\mathfrak{h}^{*}\bigotimes H^1(\Sigma(\mathcal{P}),\mathbf{k})\right)\quad\bigoplus\quad H^2(\Sigma(\mathcal{P}),\mathbf{k})\]
\end{theorem}

Observe that the necessary and sufficient conditions for a type-A Lie poset algebra to be absolutely rigid, i.e., to have no infinitesimal deformations, is the simultaneous vanishing of $(\bigwedge^2\mathfrak{h}^*\bigotimes\mathfrak{c})$, $(\mathfrak{h}^*\bigotimes H^1(\Sigma(\mathcal{P}),\mathbf{k}))$, and $H^2(\Sigma(\mathcal{P}),\mathbf{k})$. 

As we are only considering type-A Lie poset algebras corresponding to connected posets, $\mathfrak{c}$ is trivial as is shown in the lemma below. 

\begin{lemma}
If $\mathcal{P}$ is a connected poset, then $Z(\mathfrak{g}_A(\mathcal{P}))$ is trivial.
\end{lemma}
\begin{proof}
Let $z=\sum_{i\preceq j\in\mathcal{P}}z_{i,j}E_{i,j}\in Z(\mathfrak{g}_A(\mathcal{P}))$, for $z_{i,j}\in\mathbf{k}$. For $i,j\in\mathcal{P}$ such that $i\prec j$, $$\left[\frac{1}{2}(E_{i,i}-E_{j,j}),z\right]=z_{i,j}E_{i,j}+\sum_{\substack{j\neq k\in\mathcal{P} \\ i\prec k}}z'_{i,k}E_{i,k}+\sum_{\substack{k\in\mathcal{P} \\ k\prec i}}z'_{k,i}E_{k,i}+\sum_{\substack{i\neq k\in\mathcal{P} \\ k\prec j}}z'_{k,j}E_{k,j}+\sum_{\substack{k\in\mathcal{P} \\ j\prec k}}z'_{j,k}E_{j,k}=0,$$ for $z'_{i,k},z'_{k,i},z'_{k,j},z'_{j,k}\in\mathbf{k}$. Therefore, $z_{i,j}=0$, for $i,j\in\mathcal{P}$ such that $i\prec j$, and $z=\sum_{i\in\mathcal{P}}z_{i,i}E_{i,i}$. Now, note that for $i,j\in\mathcal{P}$ such that $i\prec j$, $$[z, E_{i,j}]=(z_{i,i}-z_{j,j})E_{i,j}=0,$$ i.e., $z_{i,i}=z_{j,j}$. Hence, since $\mathcal{P}$ is connected, we may conclude that $z_{i,i}=z_{j,j}$, for all $i,j\in\mathcal{P}$. Thus, since $z\in\mathfrak{sl}(|\mathcal{P}|)$, $z_{i,i}=0$, for all $i\in\mathcal{P}$. The result follows.
\end{proof}

Now, to show that $(\mathfrak{h}^*\bigotimes H^1(\Sigma(\mathcal{P}),\mathbf{k}))$ and $H^2(\Sigma(\mathcal{P}),\mathbf{k})$ are also trivial, we invoke the  Universal Coefficient Theorem, where it suffices to show that $H_n(\Sigma(\mathcal{P}),\mathbf{k})=0$ for $n=1,2$. In fact, we prove a stronger result.

\begin{theorem}\label{Nohomology}
If $\mathcal{P}$ is a connected, contact poset of height two or less, then $\Sigma(\mathcal{P})$ is contractible.
\end{theorem}

For heights zero and one, by Thoerem~\ref{thm:h1+dh2}, there are no such posets, so Theorem~\ref{Nohomology} holds vacuously. The height-two case is established by modifying the base case in the proof of Theorem 14 in \textbf{\cite{seriesA}}, which uses discrete Morse theory to show that posets of height at most two corresponding to Frobenius, type-A Lie poset algebras have contractible simplicial complexes. For the pertinent details regarding discrete Morse Theory, see \textbf{\cite{Forman}}.

%In the height-two case, we make use of the theory of discrete Morse functions \textbf{\cite{Forman},} which requires the following definitions and theorem. 
%\\*

%Let $\Sigma$ be a simplicial complex, and 
%$\alpha^{(p)}\in \Sigma$ be a $p$-simplex.

%\begin{definition}
%A function $f:\Sigma\to \mathbb{R}$ is a discrete Morse function if for every $\alpha^{(p)}\in \Sigma$
%$$|\{\beta^{(p+1)}\supset\alpha^{(p)}~|~\beta^{(p+1)}\in \Sigma, f(\beta^{(p+1)})\le f(\alpha^{(p)})\}|\le 1$$
%and 
%$$|\{\gamma^{(p-1)}\subset\alpha^{(p)}~|~\gamma^{(p-1)}\in \Sigma, f(\gamma^{(p-1)})\ge f(\alpha^{(p)})\}|\le 1.$$
%\end{definition}

%\begin{definition}
%A simplex $\alpha^{(p)}$ is critical if $$|\{\beta^{(p+1)}\supset\alpha^{(p)}~|~\beta^{(p+1)}\in \Sigma, f(\beta^{(p+1)})\le f(\alpha^{(p)})\}|=0$$
%and 
%$$|\{\gamma^{(p-1)}\subset\alpha^{(p)}~|~\gamma^{(p-1)}\in \Sigma, f(\gamma^{(p-1)})\ge f(\alpha^{(p)})\}|=0.$$
%\end{definition}

%\begin{Ex}\label{ex:DMF}
%Consider the simplicial complex $\Sigma$ illustrated in Figure~\ref{fig:Simplex} \textup(b\textup). A discrete Morse function with a single critical simplex of $v_1$ is obtained by assigning values as follows: $f(v_1)=0$, $f(e_1)=1$, $f(v_2)=2$, $f(e_2)=3$, $f(v_3)=4$, $f(e_3)=6$, and $f(f_1)=5$.
%\end{Ex}

%\begin{theorem}\label{thm:DMT}
%Suppose $\Sigma$ is a simplicial complex with a discrete Morse function. Then $\Sigma$ is homotopy equivalent to a CW complex with exactly one cell of dimension $p$ for each critical simplex of dimension $p$.
%\end{theorem}

%We are now in a position to return to the proof of Theorem~\ref{Nohomology}.

\begin{proof}[Proof of Theorem~\ref{Nohomology}]
Recall from Section~\ref{sec:mainresults} that, given a height-two poset $\mathcal{P}$, $\mathfrak{g}_A(\mathcal{P})$ is contact if and only if there exists a contact sequence $\{\mathcal{P}_i\}_{i=0}^n$ such that $\mathcal{P}_n=\mathcal{P}$; that is, $\mathfrak{g}_A(\mathcal{P})$ is contact if and only if there exists a sequence of posets $\mathcal{P}_0\subset\mathcal{P}_1\subset\hdots\subset\mathcal{P}_n=\mathcal{P}$ such that
\begin{itemize}
    \item $\mathcal{P}_0$ is of the form $\mathcal{P}(1,1,1)$ and
    \item $\mathcal{P}_i$ is obtained from $\mathcal{P}_{i-1}$ and a copy of $\mathcal{P}(1,1)$, $\mathcal{P}(2,1,1)$ or $\mathcal{P}(1,1,2)$ by applying rules $A_1$, $A_2$, $C$, $D_1$, $D_2$ or $F$ of Table~\ref{tab:h2fassem}, for $0<i\le n$.
\end{itemize}
In Theorem 11 of \textbf{\cite{seriesA}}, it is shown that posets $\mathcal{P}$ of height at most two for which $\mathfrak{g}_A(\mathcal{P})$ is Frobenius can be characterized in the exact same way, except with $\mathcal{P}_0$ of the form $\mathcal{P}(1,1,2)$ or $\mathcal{P}(2,1,1)$. In Theorem 14 of \textbf{\cite{seriesA}} the authors show that such posets have contractible simplicial complexes by recursively defining a discrete Morse function with a single critical vertex contained in the simplicial complex $\Sigma(\mathcal{P}_0)$.

Here we can use a similar argument by defining an appropriate discrete Morse function on the simplicial complex of $\mathcal{P}(1,1,1)$, illustrated below.
\begin{figure}[H]
$$\begin{tikzpicture}
\node (v7) at (-4.5,1.5) [circle, draw = black, fill = black, inner sep = 0.5mm] {};
\node (v8) at (-2.5,-1.5) [circle, draw = black, fill = black, inner sep = 0.5mm] {};
\node (v9) at (-0.5,1.5) [circle, draw = black, fill = black, inner sep = 0.5mm] {};
\draw (v7) -- (v8) -- (v9)--(v7);
\node at (-2.5,0.5) {$f_1$};
\node at (-0.5,2) {$v_3$};
\node at (-1.25,-0.25) {$e_2$};
\node at (-2.5,-2) {$v_1$};
\node at (-3.75,-0.25) {$e_1$};
\node at (-4.5,2) {$v_2$};
\node at (-2.5,2) {$e_3$};
\end{tikzpicture}$$
\caption{$\Sigma(\mathcal{P}(1,1,1))$}\label{fig:Simplex}
\end{figure}
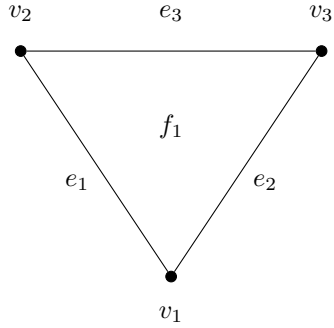

\noindent
A discrete Morse function on $\Sigma(\mathcal{P}(1,1,1))$ with a single critical simplex of $v_1$ is obtained by assigning values as follows: $f(v_1)=0$, $f(e_1)=1$, $f(v_2)=2$, $f(e_2)=3$, $f(v_3)=4$, $f(e_3)=6$, and $f(f_1)=5$. The remainder of the proof follows \textit{mutatis mutandis} to that given for Theorem 14 of \textbf{\cite{seriesA}}.
\end{proof}

We have the following immediate corollary to Theorem~\ref{Nohomology}.

\begin{corollary}
If $\mathcal{P}$ is a connected poset of height two or less for which $\mathfrak{g}_A(\mathcal{P})$ is contact, then

$$H^2(\Sigma(\mathcal{P}),\mathbf{k})=H^1(\Sigma(\mathcal{P}),\mathbf{k})=0.$$

\end{corollary}

\noindent
An application of Theorem~\ref{CG} establishes the rigidity theorem noted in the introduction.

 \begin{theorem}\label{thm:main2}
A contact, type-A Lie poset algebra corresponding to a connected poset of height zero, one, or two is absolutely rigid.
\end{theorem}

\begin{remark}  If $\mathfrak{g}_A(\mathcal{P})$ is contact, and $\mathcal{P}$ is connected and of height two or less, then by a now-classical theorem of Gerstenhaber and Schack \textup{\textbf{\cite{G3}}}, the second Hochschild cohomology group $H^2(A(\mathcal{P}),A(\mathcal{P}))$ is trivial. This implies that the \textup(associative\textup) incidence algebras corresponding to such connected posets are also rigid.
\end{remark}

\bigskip
\noindent
\textbf{Acknowledgments}  The authors thank Gil Salgado for helpful comments.

%%%%%%%%%%%%%%%%%%%%%%%%%%%%%%%%%%%%%%

\end{document}